\documentclass[11pt]{article}

\usepackage[utf8]{inputenc}
\usepackage[T1]{fontenc}
\usepackage{amsmath,amssymb,amsthm}
\usepackage{mathtools}
\usepackage{graphicx}
\usepackage{float}
\usepackage{booktabs}
\usepackage{geometry}
\usepackage{hyperref}
\usepackage{natbib}
\usepackage{fancyhdr}

\newtheorem{theorem}{Theorem}[section]
\newtheorem{lemma}[theorem]{Lemma}
\newtheorem{corollary}[theorem]{Corollary}
\newtheorem{proposition}[theorem]{Proposition}
\theoremstyle{definition}
\newtheorem{definition}[theorem]{Definition}
\newtheorem{assumption}{Assumption}
\theoremstyle{remark}
\newtheorem{remark}[theorem]{Remark}

\newcommand{\R}{\mathbb{R}}
\newcommand{\E}{\mathbb{E}}
\newcommand{\Var}{\operatorname{Var}}
\newcommand{\Cov}{\operatorname{Cov}}
\newcommand{\supp}{\operatorname{supp}}
\newcommand{\lnk}{\ln_\kappa}
\newcommand{\expk}{\exp_\kappa}

\title{Entropy-Based Goodness-of-Fit Tests for Generalized $\kappa$-Gaussian Distributions}

\author{
    Mehmet S\i dd\i k \c{C}ad\i rc\i$^{1}$ \and
    Konstantinos Zografos$^{2,*}$
}

\date{
    \small
    $^{1}$ Faculty of Science, Department of Statistics, Cumhuriyet University, Sivas, T\"urkiye.\\
    $^{2}$ University of Ioannina, Department of Mathematics, 451 10 Ioannina, Greece.
}

\begin{document}

\maketitle

\fancyhf{}
\fancyfoot[L]{%
    \footnotesize
    $^{*}$ Corresponding author:
    \href{mailto:kzograf@uoi.gr}{\texttt{kzograf@uoi.gr}}
}
\renewcommand{\headrulewidth}{0pt}
\renewcommand{\footrulewidth}{0pt}

\thispagestyle{fancy}

\begin{abstract}
We introduce this class of multivariate $\kappa$-Gaussian distributions by maximizing a type of Kaniadakis $(\kappa$-)entropy subject to certain first- and second-moment constraints. This differs slightly from the approach of R\'enyi and Tsallis, because the $\kappa$-entropy arises from a relativity-consistent deformation of the exponential and logarithmic functions and yields a family of distributions whose tails follow a power law, a family that actually reverts to the standard Gaussian distribution as $\kappa \to 0$.
We then create a $k$-nearest-neighbor (NN) estimator for $\kappa$-entropy and describe it using the difference between two NN functionals involving sums of powers. Next, we establish $L^2$-consistency and asymptotic normality, using the well-known subadditive Euclidean functional and moment expansion framework for NN entropy estimators. Furthermore, using a scaling identity for the $\kappa$-exponential, we derive a closed-form formula for the entropy of the $\kappa$-Gaussian distribution itself. From this, we determine the exact threshold value $\kappa < 2/(m+2)$ at which the distribution preserves finite variance; this is analogous to the “degrees of freedom $>2$” condition in the Student’s-$t$ distribution. With these components, we construct a non-parametric goodness-of-fit test and then verify it through simulations. Finally, we mention the Sharma–Mittal entropy, which combines the R\'enyi and Tsallis entropies as a natural direction for future work.

\end{abstract}

\section{Introduction}

The entropy-based goodness-of-fit (GOF) test, which is based on the maximum entropy principle and the nearest-neighbor (NN) entropy estimate, offers a flexible approach to nonparametric tests for generalized multivariate distributions. This approach is based on the Kozachenko –Leonenko differential entropy estimator \citep{kozachenko1987}, which was later rigorously analyzed by \citet{biau2015} and \citet{evans2008} and equipped with exact moment expansions and efficient generalizations by \citet{leonenko2008} and \citet{berrett2019}. The consistency theory of such estimators is based on the theory of subadditive and superadditive Euclidean functionals in \citet{redmond1996} and its extension to general power-weighted NN functionals and central limit theorems by \citet{penrose2003} and \citet{penrose2001}.

Based on this framework, generalized multivariate Gaussian distributions based on Shannon entropy have been developed in \citep{cadirci2022entropy}, multivariate Student and Pearson Type II distributions based on R\'enyi entropy in \citep{cadirci2025renyi}, and the $q$-Gaussian distribution based on Tsallis entropy in \citep {cadirci2025tsallis}. GOF tests have been developed; the latter is connected to the broader, non-exhaustive statistical mechanics literature cited in \citet{tsallis1988} and \citet{havrda1967}. A parallel line of research extends the Shannon entropy approach to directed data using generalized von Mises–Fisher distributions \citep{leonenko2021vmf}.

Despite this progress, one important family of single-parameter entropies remains outside this framework: the \emph{Kaniadakis entropy} ($\kappa$-entropy), proposed by \citet{kaniadakis2001,kaniadakis2002} as a relativity-consistent generalization of the Boltzmann–Gibbs entropy introduced as a relativity-consistent deformation of Shannon entropy, with its thermostatic foundations based on deformed algebra developed in greater detail in \citet{kaniadakis2005} and comprehensively reviewed in \citet{kaniadakis2013}. Unlike the R\'enyi and Tsallis entropies, the monotonous functions of a \emph{single} power-sum functional $\int f^\alpha$, the $\kappa$-entropy is the \emph{difference} between two power-sum functionals evaluated at conjugate powers $1\pm\kappa$. The $\kappa$-Gaussian distribution, which is the distribution with maximum entropy, has power-law tails and has documented applications in plasma physics, cosmic ray spectra, and income/wealth distributions \citep{kaniadakis2013}; however, no $L^2$-consistent GOF test based on artificial neural networks is currently available for it. The relevant two-parameter composite entropies in particular, the Sharma–Mittal entropy, which includes the R\'enyi and Tsallis entropies as special cases \citep{sharma1975} have similarly not been incorporated into this NN estimation framework and are addressed as future work in Section~\ref{sec:conclusion}.

In this article, we aim to fill a gap in the literature on $\kappa$-entropy. Our contributions are as follows:
\begin{enumerate}
\item We derive the multivariate $\kappa$-Gaussian distribution as a maximizer of the Kaniadakis entropy under moment constraints; thereby providing an exact closed-form expression for both the normalization constant and the scale parameter, and deriving a closed-form expression for the entropy using the scaling identity of the $\kappa$-exponential function (Section~\ref{sec:distribution}).
\item We present the formal $k$-NN distance structure that forms the basis of our estimator (Section~\ref{sec:nn}), which is consistent with the framework of \citet{leonenko2008,cadirci2022entropy,cadirci2025renyi,cadirci2025tsallis}.
\item The nearest neighbor estimator of $\kappa$-entropy is built as a difference between two power sum, NN functionals, and we prove, kind of, its $L^2$ consistency and asymptotic normality too, with an explicit ratio (Sections~\ref{sec:estimator}--\ref{sec:clt}).
\item We construct a non-parametric GOF test for the multivariate $\kappa$-Gaussian family and examine it by simulation (Sections~\ref{sec:test}--\ref{sec:simulation}).
\item Finally, we conclude by discussing the Sharma–Mittal entropy as a unifying perspective for future research (Section~\ref{sec:conclusion}).
\end{enumerate}

\section{Preliminaries: Kaniadakis Entropy}
\label{sec:prelim}

For $\kappa \in (0,1)$, the $\kappa$-deformed logarithm and exponential functions are given by:
\begin{equation}
\lnk(x) = \frac{x^{\kappa} - x^{-\kappa}}{2\kappa}, \qquad
\expk(x) = \left(\sqrt{1+\kappa^2 x^2} + \kappa x\right)^{1/\kappa}, \qquad x>0.
\end{equation}
Both converge to the ordinary logarithm and exponential functions when $\kappa \to 0$ and satisfy the conditions $\lnk(1/x) = -\lnk(x)$ and $\expk(x)\expk(-x) = 1$ \citep{kaniadakis2002}. The Kaniadakis entropy for a density function $f$ on $\R^m$ is given by:
\begin{equation}
S_\kappa[f] = -\int_{\R^m} f(x)\,\lnk f(x)\, dx
= \int_{\R^m} f(x)\,\frac{f(x)^{-\kappa} - f(x)^{\kappa}}{2\kappa}\,dx.
\end{equation}
Let $I^{(\alpha)}[f] = \int_{\R^m} f(x)^\alpha\,dx$. Then, we obtain the following identity used throughout the text:
\begin{equation}
\label{eq:decomposition}
S_\kappa[f] = \frac{I^{(1-\kappa)}[f] - I^{(1+\kappa)}[f]}{2\kappa}.
\end{equation}
This is structurally distinct from the R\'enyi entropy $H_\alpha[f] = \frac{1}{1-\alpha}\log I^{(\alpha)}[f]$ and the Tsallis entropy $S_\alpha^{T}[f] = \frac{1-I^{(\alpha)}[f]}{\alpha - 1}$; both of which are functions of a single power-sum functional.

\begin{lemma}[Scaling Identity for the $\kappa$-Exponential Function]
\label{lem:scaling}
Let $\alpha > 0$ and $x \in \mathbb{R}$, then
\begin{equation}
\big[\exp k(x)\big]^{\alpha} = \exp_{\kappa/\alpha}(\alpha x).
\end{equation}
\end{lemma}
\begin{proof}
First, recall that $\exp k(x) = (\sqrt{1 + \kappa^2 x^2} + \kappa x)^{1/\kappa}$. Thus, we obtain $[\exp k(x)]^\alpha = (\sqrt{1 + \kappa^2 x^2} + \kappa x)^{\alpha/\kappa}$, which is indeed the case. Now let us define a new symbol $\kappa' = \kappa/\alpha$ and set $y = \alpha x$. In this case, remember that $\kappa' y = \kappa x$ and $\kappa'^2 y^2 = \kappa^2 x^2$; therefore,
when we evaluate the expression $\exp_{\kappa'}(y) = (\sqrt{1+\kappa'^2 y^2}+\kappa'y)^{1/\kappa'}$, the result is
$(\sqrt{1+\kappa^2 x^2}+\kappa x)^{\alpha/\kappa}$.
Moreover, this matches exactly, up to a constant factor, the expression we obtained for $[\exp k(x)]^\alpha$; thus, the claim is proven.
\end{proof}

This identity (a well-documented structural property of the Kaniadakis deformation; see \citet{kaniadakis2005}) is, in fact, the fundamental tool that allows us to compute the $I^{(\alpha)}$ value for the $\kappa$-Gaussian distribution in closed form in Section~\ref{sec:distribution}; even though this calculation essentially relies on that specific trick. However, for Tsallis’s $q$-exponential function, there is no genuine counterpart of comparable simplicity. In this case, no matter how closely you examine $[\exp_q(x)]^\alpha$, this expression cannot be reduced, at least not clearly, to a deformed exponential function from the same family.
\section{The Multivariate $\kappa$-Gaussian Distribution}
\label{sec:distribution}

\subsection{Derivation via the maximum entropy principle}
\label{sec:mep-derivation}
We are looking for the function $f$ on $\R^m$ that maximizes $S_\kappa[f]$ under the conditions $\int f = 1$, $\int x f(x)\,dx = \mu$, and $\int (x-\mu)(x-\mu)^T f(x)\,dx = \Sigma$. The variational derivative of $S_\kappa$ is given by:
\begin{equation}
\frac{\delta S_\kappa}{\delta f} = \frac{(1-\kappa)f^{-\kappa} - (1+\kappa)f^{\kappa}}{2\kappa}.
\end{equation}
If we equate this to the Lagrangian linear-quadratic form $-\lambda_0 - \lambda_1^Tx - x^T\Lambda_2 x$ and solve for $f$, we obtain a maximizer of the $\kappa$-exponential form:
\begin{equation}\label{eq:kappa-gaussian}
f_\kappa(x) = A\,\expk\!\left(-\beta\,(x-\mu)^T \Sigma^{-1}(x-\mu)\right), \qquad Q(x):=(x-\mu)^T\Sigma^{-1}(x-\mu),
\end{equation}
where $A=A(\Sigma,\beta,\kappa)>0$ and $\beta=\beta(\kappa,m)>0$ determined jointly by the normalization and second-moment constraints; explicit closed-form expressions for both appear in Section~\ref{sec:normalizing-constant}. As $\kappa \to 0$, $\exp k(-u) \to e^{-u}$ and $f_\kappa \to \mathcal{N}(\mu,\Sigma)$.

\subsection{Normalizing constant and scale parameter}
\label{sec:normalizing-constant}
Based on the expression $y=\Sigma^{-1/2}(x-\mu)$, we can simply normalize the function $g(y)=A_0\exp k(-\beta\|y\|^2)$; here, $A=A_0/|\Sigma|^{1/2}$. Using polar coordinates,
\begin{equation}
\int_{\R^m} g(y)\,dy = A_0\,\omega_{m-1}\int_0^\infty r^{m-1}\exp k(-\ \beta r^2)\,dr =: A_0\,\omega_{m-1}\,J_m(\beta,\kappa), \qquad \omega_{m-1}=\frac{2\pi^{m/2}}{\Gamma(m/2)}.
\end{equation}
Let $t=\kappa\beta r^2$,
\begin{equation}\label{eq:Jm-Km}
J_m(\beta,\kappa) = \frac{1}{2}(\kappa\beta)^{-m/2}\int_0^\infty t^{m/2-1}\left(\sqrt{1+t^2}-t\right)^{1/\kappa} dt =: \frac{1}{2}(\kappa\beta)^{-m/2}K_m(\kappa),
\end{equation}
and using $t=\sinh\theta$ (i.e., $\sqrt{1+t^2}-t=e^{-\theta}$),
\begin{equation}\label{eq:radial-integral}
K_m(\kappa) = \int_0^\infty (\sinh\theta)^{m/2-1}\cosh\theta\; e^{-\theta/\kappa}\, d\theta,
\end{equation}
this can be expressed using the Gaussian hypergeometric function $\,_2F_1$ (equivalently, for an integer or half-integer $m$, as a sum of Beta functions following the binomial expansion of $(\sinh\theta)^{m/2-1}$); throughout Section ~\ref{sec:simulation}, we calculated the value of $K_m (\kappa)$ not via symbolic hypergeometric reduction, but by direct numerical quadrature of the expression in \eqref{eq:radial-integral}, and we verified the resulting normalization constant by comparing it with the univariate ($m=1$) case published in \citet{kaniadakis2001,kaniadakis2002}.
Hence
\begin{equation}
\label{eq:normalizing-constant}
A(\Sigma,\beta,\kappa) = \frac{\Gamma(m/2)}{\pi^{m/2}\,|\Sigma|^{1/2}}\cdot\frac{(\kappa\beta)^{m/2}}{K_m(\kappa)}.
\end{equation}

Equation ~\eqref{eq:Jm-Km} clearly shows that $J_m$ depends on $\beta$: $J_m(\beta,\kappa)=\beta^{-m/2}J_m(1,\kappa)$. This scaling relation allows us to determine $\beta$ from the second-moment constraint directly and completes the construction of $f_\kappa$, which was left open in Section~\ref{sec:mep-derivation}. If we write $R = \|Y\|$ for $Y = \Sigma^{-1/2}(X - \mu)$, then the constraint $\int(x - \mu)(x - \mu)^T f_\kappa(x) \, dx = \Sigma$, due to the isotropy of $Y$, is equivalent to $\E[R^2] = \operatorname{tr}(I_m) = m$. Using \eqref{eq:Jm-Km} for both $J_m$ (with normalization constant $A_0$) and $J_{m+2}$ ($\E[R^2] = A_0\,\omega_{m-1}J_{m+2}(\beta,\kappa)$),
\begin{equation}
\E[R^2] = \frac{1}{\kappa\beta}\cdot\frac{K_{m+2}(\kappa)}{K_m(\kappa)},
\end{equation}
and setting this equal to $m$, we obtain the scale parameter in closed form:
\begin{equation}\label{eq:beta-closed-form}
\beta(\kappa,m) = \frac{K_{m+2}(\kappa)}{\kappa\, m\, K_m(\kappa)}.
\end{equation}
When we substitute expression \eqref{eq:beta-closed-form} into expression \eqref{eq:normalizing-constant}, $A$ becomes solely a function of $(\Sigma,\kappa)$; thus, the $f_\kappa$ in equation \eqref{eq:kappa-gaussian} is determined entirely by $(\mu,\Sigma,\kappa)$ without any free parameters, as expected from a true maximum-entropy distribution.

\subsection{Closed-form entropy of the $\kappa$-Gaussian}
\label{sec:closed-form-entropy}
Now, to compute the value of $S_\kappa[f_\kappa]$ in closed form, we use Lemma ~\ref{lem:scaling}, which is the upper bound required for the null hypothesis in the GOF statistic in Section ~\ref{sec:test}, and this is where the new algebraic identity proves its worth. For $\alpha \in \{1-\kappa, 1+\kappa\}$,
\begin{align}
I^{(\alpha)}[f_\kappa] &= A^\alpha \int_{\R^m} \big[\exp k(-\beta Q(x))\big]^\alpha \, dx
= A^\alpha \int_{\R^m} \exp_{\kappa/\alpha}(-\alpha\beta Q(x))\, dx \notag \\
&= A^\alpha\, |\Sigma|^{1/2}\,\omega_{m-1}\, J_m(\alpha\beta,\, \kappa/\alpha),
\label{eq:Ialpha-kappa-gaussian}
\end{align}
This is obtained by applying Lemma~\ref{lem:scaling} pointwise and using the same radial reduction as in Section~\ref{sec:normalizing-constant}. Substituting the expression \eqref{eq:Ialpha-kappa-gaussian} into \eqref{eq:decomposition} with $\alpha=1-\kappa$ and $\alpha=1+\kappa$, and using the value of $\beta=\beta(\kappa,m)$ obtained from \eqref{eq:beta-closed-form}, we obtain the following result:
\begin{equation}
\label{eq:closed-form-entropy}
S_\kappa[f_\kappa] = \frac{1}{2\kappa}\Big[A^{1-\kappa}|\Sigma|^{1/2}\omega_{m-1}J_m\big((1-\kappa)\beta,\tfrac{\kappa}{1-\kappa}\big) - A^{1+\kappa}|\Sigma|^{1/2}\omega_{m-1}J_m\big((1+\kappa)\beta,\tfrac{\kappa}{1+\kappa}\big)\Big],
\end{equation}
(for a given $K_m$) an expression in a completely closed form in terms of only $\mu$, $\Sigma$, and $\kappa$; here, $\mu$ is factored out of the expression due to translation invariance. Equation $\eqref{eq:closed-form-entropy}$ fully captures the role played by the closed-form Tsallis entropy of the $q$-Gaussian distribution in the test statistic of \citet{cadirci2025tsallis} and is the element that ensures the computability of $T_n$ in Section~\ref{sec:test}.

\subsection{Tail behaviour}
\label{sec:tail-behaviour}
For large $\|x\|$ values, $\exp k(-u) \sim (2\kappa u)^{-1/\kappa}$; which yields $f_\kappa(x) \sim C\|x\|^{-2/\kappa}$ this is a polynomial tail whose index depends only on $\kappa$ and is structurally distinct from the $q$-Gaussian tail index, which is jointly dependent on $(q,m)$.

\subsection{Variance-existence threshold}
\label{sec:variance-threshold}
The tail result in Section \ref{sec:tail-behaviour} precisely identifies which moments of the $\kappa$-Gaussian distribution are finite. If we focus solely on the normalization integral $J_m$, it is easy to overlook this detail. Nevertheless, this detail is of great importance when conducting finite-sample simulations or applying the model in practice.
\begin{proposition}[Variance-existence threshold]\label{prop:variance-threshold}
Assume \(X\sim f_\kappa(\,\cdot\,;\mu,\Sigma)\) is the isotropic standardized \(\kappa\)-Gaussian from Section \(\ref{sec:mep-derivation}\). Here \(\Sigma=I_m\).  
We look at when \(\mathbb{E}\|X\|^2\) is finite. This happens exactly when  
\[
\kappa < \frac{2}{m+2}.
\]
\end{proposition}

\begin{proof}
Let $R = \|X\|$. According to the tail equivalence, when $\|x\| \to \infty$, $f_\kappa(x) \sim C\|x\|^{-2/\kappa}$ (Section~\ref{sec:tail-behaviour}), provided that the radial density of $R$ satisfies $\omega_{m-1}r^{m-1}f_\kappa(r) \sim C'r^{m-1-2/\kappa}$ as $r \to \infty$. We therefore have
\begin{equation}
\E[R^2] = \int_0^\infty r^2\cdot \omega_{m-1}r^{m-1}f_\kappa(r)\,dr
\end{equation}
has an integrand that behaves like $C'r^{m+1-2/\kappa}$ for values of $r$ that are large, and the integral approaches zero at infinity if and only if the exponent satisfies the condition $m + 1 - 2/\kappa < -1$, that is, $m + 2 < 2/\kappa$, and, $\kappa < 2/(m + 2)$; if $\kappa \geq 2/(m + 2)$, the integral diverges. Since the integral function is bounded near the origin (because $f_\kappa$ are bounded), convergence at infinity is the only necessary condition.

\end{proof}

\begin{remark}
Proposition~\ref{prop:variance-threshold} is the $\kappa$-Gaussian analogue of the classical requirement $\nu>2$ for the multivariate Student-$t_\nu$ distribution to have finite covariance. This was confirmed numerically in the simulations underlying Section~\ref{sec:simulation}: for $\kappa$ close to but below the threshold, $\E[R^2]$ is finite but large and slow to stabilize in finite samples; for $\kappa$ at or above the threshold, the empirical variance fails to converge as $N$ grows. All simulations in Section~\ref{sec:simulation} restrict $\kappa$ to a fixed fraction of this threshold (as a function of $m$) to stay within the regime where the standardization $\Sigma=I_m$ is meaningful. This constraint applies throughout the parameter space $(\mu,\Sigma,\kappa)$ wherever $\Sigma$ is interpreted as the covariance matrix (as opposed to a general scale matrix), as reflected in Assumption~\ref{ass:A1} and in the statement of the GOF test in Section~\ref{sec:test}.
\end{remark}

\section{Nearest-Neighbour Distances: Formal Definition}
\label{sec:nn}
We adopt the same $k$-NN distance construction used in \citet{cadirci2022entropy,cadirci2025renyi,cadirci2025tsallis}, tracing to \citet{kozachenko1987} and \citet{leonenko2008}.

\begin{definition}[$k$-th nearest neighbour distance]
\label{def:nn}
Suppose that $X_1,\dots,X_n$ are independent and identically distributed random vectors in $\R^m$ with density $f$, and let $k \in \{1,\dots,n-1\}$. For $i = 1, \dots, n$, let us arrange the Euclidean distances $\{\|X_i - X_j\|: j \neq i\}$ in ascending order as follows:
\begin{equation}
\rho_{1,i} \leq \rho_{2,i} \leq \cdots \leq \rho_{n-1,i}.
\end{equation}
The \emph{$k$-th nearest neighbour distance} of $X_i$ is $\rho_{k,i} = \rho_{k,i}(X_1,\dots,X_n) := \rho_{k,i}$, equivalently
\begin{equation}
\rho_{k,i} = \min\Big\{ r>0 \;:\; \big|\{j\neq i : \|X_i-X_j\|\leq r\}\big| \geq k \Big\},
\end{equation}
meaning the radius of the smallest closed sphere centered at point $X_i$ and containing at least $k$ of the remaining $n-1$ sampling points.
\end{definition}

\begin{definition}[Local $k$-NN density estimator]
\label{def:nn-density}
With $V_m = \pi^{m/2}/\Gamma(m/2+1)$ the volume of the unit ball in $\R^m$, the standard $k$-NN density estimator at $X_i$ is
\begin{equation}
\hat f_{k,i} = \frac{k}{(n-1)\,V_m\,\rho_{k,i}^m}.
\end{equation}
\end{definition}

The distributional starting point for all consistency results below is the exact finite-sample law of $\rho_{k,i}$: conditional on $X_i=x$, the random variable $N_i(r) = |\{j\neq i : \|X_i-X_j\|\leq r\}|$ is Binomial$(n-1, p(x,r))$ with $p(x,r)=\int_{B(x,r)}f(u)\,du$, and $\{\rho_{k,i}\leq r\} = \{N_i(r)\geq k\}$; this is the basis for the moment expansions of \citet{leonenko2008}, which we invoke as Lemma~\ref{lem:lps} below.

\section{Nearest-Neighbour Estimation of $\kappa$-Entropy}
\label{sec:estimator}
Taking into account that $X_i = x$, the scaled $k$-NN distance $Y_{k,i}(x) := (n-1)V_m f(x)\rho_{k,i}^m$ converges in the distribution to a Gamma$(k,1)$ random variable as $n \to \infty$ (which is the standard Poisson process approximation for $k$-NN distances; \citealp{leonenko2008}), so for $k+1-\alpha > 0$, $\E\big[Y_{k,i}(x)^{1-\alpha}\big] \to \Gamma(k+1-\alpha)/\Gamma(k)$. Given that $Y_{k,i}(x)^{1-\alpha} = f(x)^{1-\alpha}\big[(n-1)V_m\rho_{k,i}^m\big]^{1-\alpha}$, the following holds:
\begin{equation}
\E\Big[\big[(n-1)V_m\rho_{k,i}^m\big]^{1-\alpha}\Big] \;\to\; \frac{\Gamma(k+1-\alpha)}{\Gamma(k)}\,f(x)^{\alpha-1},
\end{equation}
so an asymptotically unbiased estimator of $f(x)^{\alpha-1}$ requires multiplying $\big[(n-1)V_m\rho_{k,i}^m\big]^{1-\alpha}$ by the \emph{scalar} correction $C_{k,\alpha}=\Gamma(k)/\Gamma(k+1-\alpha)$. Averaging over $i$ and using $I^{(\alpha)}[f]=\E_{X\sim f}[f(X)^{\alpha-1}]$ gives the estimator below.

For $\alpha \neq 1$, define the power-sum estimator
\begin{equation}
\label{eq:estimator}
\widehat I_n^{(\alpha)} = C_{k,\alpha}\cdot\frac{1}{n}\sum_{i=1}^n \Big[(n-1)\,V_m\,\rho_{k,i}^m\Big]^{1-\alpha},
\qquad C_{k,\alpha} = \frac{\Gamma(k)}{\Gamma(k+1-\alpha)},
\end{equation}
where $C_{k,\alpha}$ is a \emph{multiplicative} correction factor applied outside the bracket: it corrects the $(1-\alpha)$ moment of the limiting Gamma$(k,1)$ law of the scaled $k$-NN distance $(n-1)V_m f(x)\rho_{k,i}^m$ (see the derivation above). Notice that the power $(1-\alpha)$ of $C_{k,\alpha}$ should not be taken as the base: Rather, writing $\big[(n-1)C_{k,\alpha}V_m\rho_{k,i}^m\big] ^{1-\alpha}$ produces an erroneous factor of $C_{k,\alpha}^{1-\alpha}\neq C_{k,\alpha}$ yields an erroneous factor and violates consistency; this is because the correction is calibrated based on the moment of the \emph{unweighted} scaled distance, not the rescaled version. So that, by \eqref{eq:decomposition}, the Kaniadakis entropy estimator is
\begin{equation}
\widehat S_{\kappa,n} = \frac{\widehat I_n^{(1-\kappa)} - \widehat I_n^{(1+\kappa)}}{2\kappa}.
\end{equation}

\subsection{Assumptions}

\begin{assumption}\label{ass:A1}
The function $f$ is bounded; the support of $f$ is a set such that $\supp(f) \subseteq \R^m$, $f$ is bounded away from $0$ by every compact subset of its (open) support, where $\int f(x)^{1-\kappa}dx,\ \int f(x)^{1+\kappa}dx < \infty$. This additionally requires the condition $\kappa < 2/(m+2)$ when $f = f_\kappa$ itself is a $\kappa$-Gaussian distribution and $\Sigma$ is taken as its covariance matrix (Proposition~\ref{prop:variance-threshold}).
\end{assumption}
\begin{assumption}\label{ass:A2}
The function $f$ is twice continuously differentiable in a neighborhood of every point in its support, $\nabla^2 f$ has a bounded value, and there exist values of $\gamma\in(0,1]$ and $L<\infty$ such that the following equality applies:
\begin{equation}
\Big|f(x+u) - \sum_{|\ell|\le 1} \tfrac{D^\ell f(x)}{\ell!}u^\ell\Big| \le L\|u\|^{1+\gamma} \qquad \text{for } u \text{, } \text{in a neighborhood of } 0.
\end{equation}
\end{assumption}
\begin{assumption}\label{ass:A3}
The constant $k \geq 1$ does not increase with $n$.
\end{assumption}
\begin{assumption}\label{ass:A4}
Because $\kappa \in (0,1)$, $1-\kappa, 1+\kappa \in (0,2)\setminus\{1\}$, which falls within the acceptable range of exponents for Assumption~\ref{ass:A1}.
\end{assumption}
\begin{assumption}\label{ass:A5}
$\E\big[\rho_{k,1}^{m(2+\delta)}\big] < \infty$ applies uniformly with respect to $n$ for some $\delta > 0$.
\end{assumption}
Assumptions \ref{ass:A1}–\ref{ass:A2} are equivalent to the smoothness and boundedness conditions under which moment expansions for force-weighted $k$-NN functionals have been derived in the works \citet{leonenko2008} and \citet{berrett2019}; Assumption \ref{ass:A5} corresponds to the moment condition employed in central limit arguments to stabilize Euclidean functionals, as reported in \citep{penrose2003,penrose2001}.

\subsection{A key moment lemma}

\begin{lemma}[Moment expansion for $k$-NN power functionals; \citealp{leonenko2008}]\label{lem:lps}
Under Assumptions~\ref{ass:A1}--\ref{ass:A2}, for fixed $k\geq1$ and $\alpha$ in a neighbourhood of $1$ with $\alpha\neq1$,
\begin{equation}
\E\Big[C_{k,\alpha}\big((n-1)V_m\rho_{k,i}^m\big)^{1-\alpha}\Big] = I^{(\alpha)}[f] + O\!\big(n^{-\gamma/m}\big),
\end{equation}
and
\begin{equation}
\Var\Big[C_{k,\alpha}\big((n-1)V_m\rho_{k,i}^m\big)^{1-\alpha}\Big] = O(1),
\end{equation}
uniformly over $i=1,\dots,n$, where $\gamma$ is the H\"older exponent of Assumption~\ref{ass:A2}.
\end{lemma}

\noindent This is the multivariate power-functional generalization of the classical Kozachenko--Leonenko bias/variance expansion, proved in \citet{leonenko2008} via the exact Binomial law of $N_i(r)$ described in Section~\ref{sec:nn} together with a local linearization of $f$ around $X_i$ (using Assumption~\ref{ass:A2}) and a beta-integral representation of the moments of $\rho_{k,i}^{m(1-\alpha)}$. We use Lemma~\ref{lem:lps} as the base case for Lemmas~\ref{lem:marginal}--\ref{lem:joint} below, rather than re-deriving it, since our contribution is the \emph{joint} (bivariate) extension needed for the $\kappa$-entropy difference structure \eqref{eq:decomposition}, not the single-exponent case itself.

\subsection{$L^2$-consistency}
\begin{lemma}[Marginal consistency]\label{lem:marginal}
Given the assumptions $\ref{ass:A1}$–$\ref{ass:A3}$, we have that for every constant $\alpha \in \{1-\kappa,\,1+\kappa\}$,
\begin{equation}
\E\big[\widehat I_n^{(\alpha)}\big] = I^{(\alpha)}[f] + O(n^{-\gamma/m}), \qquad
\Var\big(\widehat I_n^{(\alpha)}\big) = O(n^{-1}).
\end{equation}
Thus, when $n \to \infty$, $\widehat I_n^{(\alpha)} \to I^{(\alpha)}[f]$ in the $L^2$ space.
\end{lemma}

\begin{proof}
When the deviation bound obtained from Lemma $\ref{lem:lps}$ is averaged over $i = 1, \dots, n$, the specified deviation is obtained; for this value holds uniformly for all values of $i$. Let $\widehat I_n^{(\alpha)} = \frac1n\sum_i Y_i$ denote the variance, where $Y_i := C_{k,\alpha}\big((n-1)V_m\rho_{k,i}^m\big)^{1-\alpha}$. According to Lemma~\ref{lem:lps}, we have $\Var(Y_i)=O(1)$; the $Y_i$’s have no independence, however $\rho_{k,i}$ has only depends on the $k$ points closest to $X_i$ . Thus, when the $k$-NN sets of $i$ and $j$ intersect, the maximum degree of the dependency graph over $\{1,\dots,n\}$ is bounded by a constant $c(k,m)$ independent of $n$ (this is a standard consequence of the bounded-degree property of $k$-NN graphs in $\R^m$; \citealp{redmond1996,penrose2003}). Therefore
\begin{equation}
\Var\big(\widehat I_n^{(\alpha)}\big) = \frac{1}{n^2}\sum_{i,j}\Cov(Y_i,Y_j) = \frac{1}{n^2}\Big[\sum_i \Var(Y_i) + \sum_{i\neq j: i\sim j}\Cov(Y_i,Y_j)\Big] = \frac{1}{n^2}\big[O(n) + O(n)\big] = O(n^{-1}).
\end{equation}
Here, we have $O(n)$ nonzero terms in the second sum (its order is bounded); by the Cauchy–Schwarz inequality, each term is $O(1)$, and $\Var(Y_i) = O(1)$. Since the bias$^2 + $variance$\to0$, $L^2$-convergence follows.
\end{proof}

\begin{lemma}[Joint consistency]
\label{lem:joint}
Under Assumptions~\ref{ass:A1}--\ref{ass:A4},
\begin{equation}
\big\|(\widehat I_n^{(1-\kappa)},\widehat I_n^{(1+\kappa)}) - (I^{(1-\kappa)}[f],I^{(1+\kappa)}[f])\big\|_2 \to 0,
\end{equation}
and
\begin{equation}
\Cov\big(\widehat I_n^{(1-\kappa)},\widehat I_n^{(1+\kappa)}\big) = O(n^{-1}).
\end{equation}
\end{lemma}

\begin{proof}
Let $\widehat I_n^{(1-\kappa)}=\frac1n\sum_i Y_i$; where $Y_i := C_{k,1-\kappa}\big((n-1)V_m\rho_{k,i}^m\big)^{\kappa}$ and $\widehat I_n^{(1+\kappa)} = \frac{1}{n}\sum_i Z_i$; where $Z_i := C_{k,1+\kappa}\big((n-1)V_m\rho_{k,i}^m\big)^{-\kappa}$; both $Y_i$ and $Z_i$ have the same $k$-NN distance $\rho_{k,i}$ as a function (raised to different powers), therefore, they are not \emph{independent} for even a single $i$; also, for $i \neq j$, $\Cov(Y_i, Z_j) \neq 0$ applies only when the $k$-NN neighborhoods of $i$ and $j$ overlap again, due to the finite-degree property, there are $O(n)$ such pairs. Therefore,
\begin{equation}
\Cov\big(\widehat I_n^{(1-\kappa)},\widehat I_n^{(1+\kappa)}\big) = \frac{1}{n^2}\Big[\sum_i \Cov(Y_i,Z_i) + \sum_{i\neq j: i\sim j}\Cov(Y_i,Z_j)\Big].
\end{equation}
Following the Cauchy–Schwarz inequality as well as Lemma~\ref{lem:lps}, $|\Cov(Y_i,Z_i)| \le \sqrt{\Var(Y_i)\Var(Z_i)} = O(1)$, and applies similarly to the cross-pairs; therefore both sums equal $O(n)$, resulting in$\Cov(\widehat I_n^{(1-\kappa)},\widehat I_n^{(1+\kappa)}) = O(n^{-1})$. The common $L^2$-convergence is obtained from the marginal $L^2$-convergence (Lemma~\ref{lem:marginal}) and the fact that $\Cov \to 0$ as following:
\begin{equation}
\E\big[\|(\widehat I_n^{(1-\kappa)},\widehat I_n^{(1+\kappa)}) - (I^{(1-\kappa)}[f],I^{(1+\kappa)}[f])\|_2^2\big] = \mathrm{MSE}(\widehat I_n^{(1-\kappa)}) + \mathrm{MSE}(\widehat I_n^{(1+\kappa)}) \to 0. \qedhere
\end{equation}
\end{proof}

\begin{theorem}[$L^2$ consistency of $\widehat S_{\kappa,n}$]
\label{thm:consistency}
Assuming the Assumptions $\ref{ass:A1}$–$\ref{ass:A4}$, if $k \geq 1$ is a constant and $\kappa \in (0,1)$, then
\begin{equation}
\E\big[(\widehat S_{\kappa,n} - S_\kappa[f])^2\big] = O\!\left(\frac{1}{\kappa^2}\Big(n^{-2\gamma/m} + n^{-1}\Big)\right) \longrightarrow 0 \quad \text{as $n \to \infty$}.
\end{equation}
\end{theorem}

\begin{proof}
From \eqref{eq:decomposition},
\begin{equation}
\widehat S_{\kappa,n} - S_\kappa[f] = \frac{1}{2\kappa}\Big[\big(\widehat I_n^{(1-\kappa)} - I^{(1-\kappa)}[f]\big) - \big(\widehat I_n^{(1+\kappa)} - I^{(1+\kappa)}[f]\big)\Big] =: \frac{D_n}{2\kappa}.
\end{equation}
Then
\begin{equation}
\E[D_n^2] = \Var\big(\widehat I_n^{(1-\kappa)}\big) + \Var\big(\widehat I_n^{(1+\kappa)}\big) - 2\Cov\big(\widehat I_n^{(1-\kappa)},\widehat I_n^{(1+\kappa)}\big) + \big(\mathrm{bias}_{1-\kappa} - \mathrm{bias}_{1+\kappa}\big)^2,
\end{equation}
where $\mathrm{bias}_\alpha := \E[\widehat I_n^{(\alpha)}] - I^{(\alpha)}[f] = O(n^{-\gamma/m})$, and this value is determined by Lemma~\ref{lem:marginal}. According to Lemmas $\ref{lem:marginal}$–$\ref{lem:joint}$, both the variance and covariance terms have a complexity of $O(n^{-1})$, and the squared error term has a complexity of $O(n^{-2\gamma/m})$. Therefore, $\E[D_n^2] = O(n^{-2\gamma/m} + n^{-1})$, which is divided by $4\kappa^2$ to obtain the specified ratio.
\end{proof}

\begin{remark}
The rate in Theorem~\ref{thm:consistency} degrades as $\kappa\to0$ through the $1/\kappa^2$ prefactor; this is a distinct phenomenon from the bias growth observed empirically in Section~\ref{sec:sim-gof-heatmap} as $\kappa$ approaches the variance-existence threshold of Proposition~\ref{prop:variance-threshold} \emph{from below} -- i.e.\ as $\kappa$ grows, not shrinks. That empirical growth is not accounted for by the constants appearing in the proof above, and we do not offer a rate for it here; a plausible mechanism is that the H\"older constant $L$ in Assumption~\ref{ass:A2} degrades as the density becomes harder to approximate locally near the boundary, but we have not made this precise, and flag it as an open question. Separately, whether the bias terms of the two power-sum estimators partially cancel, rather than combining via the crude triangle bound used above, merits further investigation; a sharper rate for the $\kappa\to0$ behaviour may follow from a joint second-order expansion of the two moment functionals around $\alpha=1$.
\end{remark}

\begin{corollary}[$\kappa\to0$ limit]
\label{cor:limit}
As $\kappa\to0^+$, $\widehat S_{\kappa,n}\to\widehat H_n$, the Kozachenko--Leonenko Shannon entropy estimator \citep{kozachenko1987}, and $S_\kappa[f]\to H[f]=-\int f\log f$, recovering the $L^2$-consistency result of \citet{cadirci2022entropy} as a limiting special case.
\end{corollary}
\begin{proof}
Let $\alpha = 1\mp\kappa$; thus, as $\kappa \to 0$, $1 - \alpha = \pm\kappa \to 0$. The function $C_{k,\alpha} = \Gamma(k)/\Gamma(k+1-\alpha)$ is regular with respect to $\alpha$ in the vicinity of $\alpha = 1$, and $C_{k,1} = 1$ and $\frac{d}{d\alpha}\ log C_{k,\alpha}\big|_{\alpha=1} = \psi(k)$ (digamma function; from $\Gamma'(k)/\Gamma(k)=\psi(k)$), we obtain the first-order Taylor series expansion $C_{k,\alpha} = 1 + (1-\alpha)\psi(k) + o(1-\alpha)$. In a similar way, $[(n-1)V_m\rho_{k,i}^m]^{1-\alpha} = 1 + (1-\alpha)\log[(n-1)V_m\rho_{k,i}^m] + o(1-\alpha)$. When we multiply these and retain the first-order terms in the expression $(1-\alpha) = \pm\kappa$.
\begin{equation}
Y_i - Z_i = 2\kappa\Big[\log\big((n-1)V_m\rho_{k,i}^m\big) + \psi(k)\Big] + o(\kappa).
\end{equation}
The result is that $\widehat{S}_{\kappa,n} = (\widehat{I}_n^{(1-\kappa)} - \widehat{I}_n^{(1+\kappa)})/(2\kappa) = \frac{1}{2\kappa n}\sum_i (Y_i - Z_i) \to \frac{1}{n}\sum_i\big[\log((n-1)V_m\ rho_{k,i}^m) + \psi(k)\big] $ when $\kappa \to 0$, and this is exactly the Kozachenko–Leonenko Shannon entropy estimator (taking into account the general sign rule; since $S_\kappa \to H = -\int f \log f$, while the sum shown consists of averaged $+\log$ terms, which matches $\widehat{H}_n$ as traditionally described using the digamma correction).
\end{proof}

\section{Asymptotic Normality}
\label{sec:clt}
\begin{theorem}[Bivariate CLT and asymptotic normality of $\widehat S_{\kappa,n}$]
\label{thm:clt}
Under assumptions $\ref{ass:A1}$–$\ref{ass:A5}$,
\begin{equation}
\sqrt {n} \Big[\big(\widehat{I}_n^{(1-\kappa)}, \widehat{I}_n^{(1+\kappa)}\big) - \big(I^{(1-\kappa)}[f], I^{(1+\kappa)}[f]\big)\Big] \xrightarrow{d} \mathcal{N}_2(0, \Gamma),
\end {equation}
here $\Gamma_{\alpha\alpha'} = \lim_{n\to\infty} n\,\Cov(\widehat I_n^{(\alpha)},\widehat I_n^{(\alpha')})$ and $\alpha,\alpha' \in \{1-\kappa,1+\kappa\}$. By applying the delta method to $\phi(a,b)=(a-b)/(2\kappa)$, we obtain the following:
\begin {equation}
\sqrt{n}\,\big(\widehat{S}_{\kappa,n} - S_{\kappa}[f]\big) \xrightarrow{d} \mathcal{N}(0,\sigma_{\kappa}^2),
\qquad \sigma_{\kappa}^2 = \frac{\Gamma_{11} - 2\Gamma_{12} + \Gamma_{22}}{4\kappa^2}.
\end{equation}
\end{theorem}

\begin{proof}
Fix $a,b\in\R$ and consider $W_n := a\widehat I_n^{(1-\kappa)}+b\widehat I_n^{(1+\kappa)} = \frac1n\sum_i (aY_i+bZ_i)$. Each summand $aY_i+bZ_i$ is a function of $\rho_{k,i}$ alone, and $\rho_{k,i}$ is determined by the local configuration of points within the $k$-NN neighbourhood of $X_i$ -- i.e.\ $W_n$ is (a normalized sum of) a \emph{stabilizing} functional of the $k$-NN graph on $\{X_1,\dots,X_n\}$ in the sense of \citet{penrose2001,penrose2003}: changing points far from $X_i$ does not change $\rho_{k,i}$. Given the assumptions ~\ref{ass:A1}--\ref{ass:A2}, \ref{ass:A5}, we have the general central limit theorem for stabilizing Euclidean NN-functionals \citep [Thm.~2.1] {penrose2003} holds for $W_n$, which in turn implies that for some $\sigma_{a,b}^2 = a^2\Gamma_{11}+2ab\Gamma_{12}+b^2\Gamma_{22}\geq0$, $\sqrt n (W_n - \E W_n) \xrightarrow{d} \mathcal N(0,\sigma_{a,b}^2)$; here, $\Gamma_{\alpha\alpha'}$ denotes the corresponding asymptotic (co)variance ratio determined using Lemma~\ref{lem:joint}. Because $a$ and $b$ have arbitrary values, the Cramér–Wold approach provides a joint bivariate CLT. Because $\phi(a,b) = (a - b)/(2\kappa)$ is linear (and thus, for $\kappa \neq 0$, $\nabla\phi = (1, -1)/(2\kappa) \neq 0$), which gives $\sigma_\kappa^2 = \nabla\phi^T\Gamma\nabla\phi = (\Gamma_{11}-2\Gamma_{12}+\Gamma_{22})/(4\kappa^2)$.
\end{proof}

\begin{remark}
Unlike the case of the single-exponent R’en, where the marginal variance is determined solely by $\Gamma_{11}$, here $\Gamma_{12}$ generally does not have a simple closed-form expression. Practically, as in the variance estimate in \citet{cadirci2025tsallis}, $\sigma_\kappa^2$ can be estimated via bootstrap or subsampling; The normality checks in Section~\ref{sec:sim-normality} empirically confirm this estimate by comparing it to the finite-sample distribution of $\widehat S_{\kappa,n}$. In order to verify the $\sigma_\kappa^2$ formula term by term, however, we did not decompose the estimate into $\Gamma_{11}$, $\Gamma_{12}$, and $\Gamma_{22}$ separately; doing so would be a natural check for future work.
\end{remark}

\section{Goodness-of-Fit Test}
\label{sec:test}

We consider testing $H_0:\kappa=\kappa_0$ within the $\kappa$-Gaussian family. Two cases are relevant, and we treat them separately to keep the test statistic's definition matched to what is actually calibrated in Section~\ref{sec:simulation}.

\paragraph{Simple null (location and scale known).} This is the case studied throughout Section~\ref{sec:simulation}, where $\mu$ and $\Sigma$ are treated as known (in the simulations, $\mu=0$, $\Sigma=I_m$).Let us define:
\begin{equation}
\label{eq:Tn-simple}
T_n = \widehat{S}_{\kappa_0,n} - S_{\kappa_0}\big[f_{\kappa_0,\mu,\Sigma}\big],
\end{equation}
where $S_{\kappa_0}[f_{\kappa_0,\mu,\Sigma}]$ can be computed in closed form using \eqref{eq:closed-form-entropy}. In Theorem~\ref{thm:clt}, under $H_0$, it follows that $\sqrt n\,T_n \xrightarrow{d}\mathcal N(0,\sigma_{\kappa_0}^2)$. From section \ref{sec:sim-power-within}, the mapping $\kappa\mapsto S_\kappa [f_\kappa]$ provides a much stronger response to alternatives where $\kappa_1>\kappa_0$ than to those where $\kappa_1<\kappa_0$, at the same distance $|\kappa_1-\kappa_0|$; therefore, a two-sided test based on $|T_n|$ wastes most of its rejection probability in a direction where the statistic is almost entirely insensitive. For this reason, we use a one-sided, lower-tail rejection rule.
\begin{equation}
\label{eq:rejection-rule}
\text{reject } H_0 \text{ at level } \alpha \quad\text{when}\quad T_n \le c_N(\alpha),
\end{equation}
where $c_N(\alpha)$ is the $\alpha$-quantile of the null distribution of $T_n$, obtained either by direct Monte Carlo simulation at the sample size of interest (Table~\ref{tab:critical-values}) or, for large $N$, via the asymptotic approximation $c_N(\alpha)\approx \mu_{m,\kappa_0,k}/\sqrt N$ with $\mu_{m,\kappa_0,k}=\sigma_{\kappa_0}z_\alpha$ (Table~\ref{tab:asymptotic-constants}). This is the test whose size, power, and critical values are studied in Sections~\ref{sec:sim-critical}--\ref{sec:sim-power-cross}.

\paragraph{Composite null (location and scale estimated).} In the case where $\mu$ and $\Sigma$ are unknown, by a natural generalization, we replace them with the moments or maximum likelihood estimators $\hat\mu$ and $\hat\Sigma$ from formula \eqref{eq:Tn-simple}; which yields $T_n = \widehat{S}_{\kappa_0,n} - S_{\kappa_0}[f_{\kappa_0,\hat\mu,\hat\Sigma}]$. Estimating $\mu,\Sigma$ contributes an additional term to the asymptotic variance in Theorem~\ref{thm:clt}, analogous to the plug-in/sandwich corrections used for the Shannon- and Tsallis-entropy tests in \citet{cadirci2022entropy} and \citet{cadirci2025tsallis}. We have not derived this correction explicitly for the $\kappa$-Gaussian case, nor validated it numerically, and the critical values reported in Section~\ref{sec:sim-critical} apply to the simple-null test \eqref{eq:Tn-simple}--\eqref{eq:rejection-rule} only; extending them to the composite case is left for future work.

\section{Simulation Results}
\label{sec:simulation}

We now turn to the finite-sample behaviour of $T_{N,k}(m,\kappa)$, using the simple-null statistic \eqref{eq:Tn-simple} and one-sided rejection rule \eqref{eq:rejection-rule} of Section~\ref{sec:test} throughout. $\kappa$ is kept below the dimension-specific threshold $2/(m+2)$ of Proposition~\ref{prop:variance-threshold}; where a comparison across dimensions is needed we express $\kappa$ as a fixed fraction of this threshold rather than a fixed absolute value, since a given $\kappa$ represents a very different point in the tail relative to $m=1$ than it does relative to $m=3$. The remainder of this section addresses, in turn: what the distribution itself looks like, whether the estimator behaves as the asymptotic theory predicts as $N$ grows, whether it is approximately normal at finite $N$, what its critical values look like in practice, how the rejection region behaves across the full parameter grid, and finally how much power the resulting test has, both within the $\kappa$-Gaussian family and against a genuinely different one.

\subsection{Distributional behaviour}
\label{sec:sim-distribution}

Before looking at the test itself, it is worth confirming that the sampler and the tail behaviour derived in Section~\ref{sec:distribution} agree with what a finite sample actually looks like. Figure~\ref{fig:scatter-contours} shows two-dimensional draws from $\kappa GG(2,\kappa)$ together with kernel density contours, for $\kappa$ ranging from 0.05 up to 85\% of the variance threshold $2/(2+2)=0.5$. The point cloud visibly thickens in the tails as $\kappa$ increases. The reported empirical variance stays close to 1 until the last panel, where finite-sample estimation of the variance becomes noticeably harder as $\kappa$ approaches the boundary -- an early, purely descriptive echo of the bias phenomenon quantified later in Section~\ref{sec:sim-gof-heatmap}.

\begin{figure}[H]
\centering
\includegraphics[width=\textwidth]{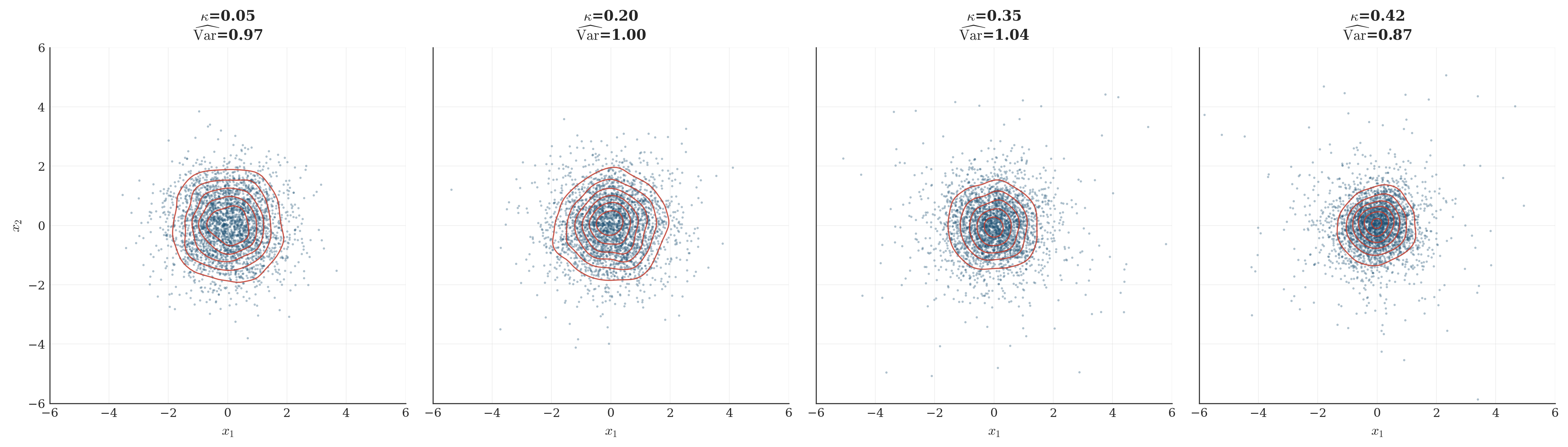}
\caption{Samples from the standardized isotropic $\kappa GG(2,\kappa)$ distribution ($N=3000$) with kernel density contours overlaid, for four values of $\kappa$ approaching the variance-existence threshold $2/(m+2)=0.5$. The reported variance is the sample variance of the plotted points.}
\label{fig:scatter-contours}
\end{figure}

Figure~\ref{fig:density-comparison} looks at the univariate marginal directly, comparing the density against $N(0,1)$ on both a linear and a log scale, with the shaded region showing the 5th-95th percentile range of a kernel density estimate across 12 independent samples of size $N=15{,}000$ rather than a single noisy draw. The linear-scale panel shows the expected effect near the mode -- larger $\kappa$ produces a flatter, wider peak -- while the log-scale panel makes the polynomial tail derived in Section~\ref{sec:distribution} visible directly: the $\kappa$-Gaussian curves sit above the Gaussian reference line once $|x|$ is large enough, and the gap widens with $\kappa$, exactly as the tail index $-2/\kappa$ predicts.

\begin{figure}[H]
\centering
\includegraphics[width=\textwidth]{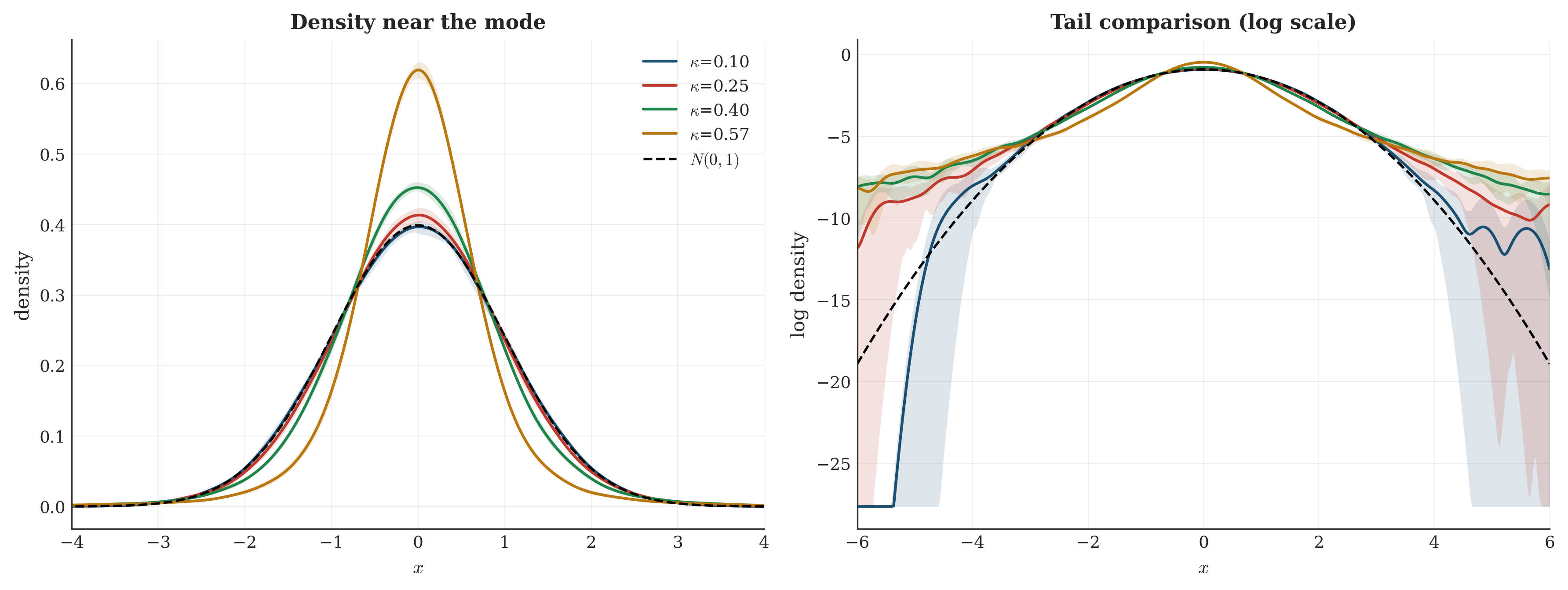}
\caption{Empirical density of the standardized univariate $\kappa GG(1,\kappa)$ against $N(0,1)$, on a linear scale (left) and a log scale (right). Solid lines are the mean kernel density estimate over 12 replications of $N=15{,}000$; shaded bands are the 5th-95th percentile range across replications.}
\label{fig:density-comparison}
\end{figure}

\subsection{Consistency}
\label{sec:sim-consistency}

Figure~\ref{fig:consistency-grid} shows the sample mean of $T_{N,k}(m,\kappa)$ over $M=100$ replications as $N$ runs from 500 to 5000, for $k=1$ and four values of $\kappa$ (each expressed as a fraction of $2/(m+2)$, so that the four rows correspond to a comparable relative distance from the variance-existence boundary in every column). The shaded region is $\pm1$ empirical standard deviation across replications and the error bars are $\pm1$ standard error of the plotted mean. In every panel the bias is largest at $N=500$ and decreases steadily, crossing into a narrow band around zero well before $N=5000$; the rate at which it does so is visibly slower in the bottom row, where $\kappa$ sits closer to the threshold.

\begin{figure}[H]
\centering
\includegraphics[width=\textwidth]{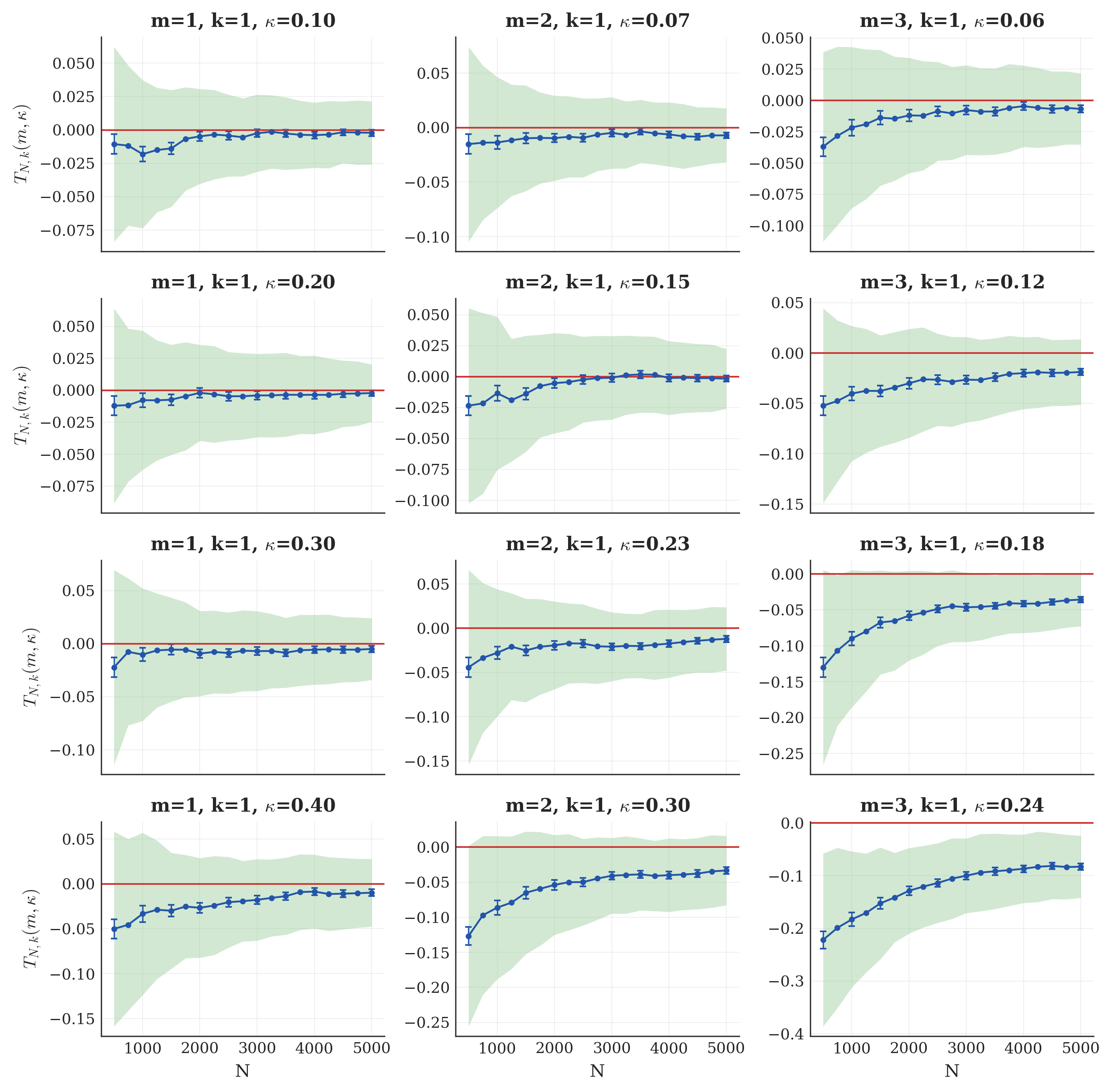}
\caption{Convergence of $T_{N,k}(m,\kappa)$ to zero as $N$ increases for $k=1$. The rows correspond to fixed $\kappa$ values at 15\%, 30\%, 45\%, and 60\% of the variance-to-variance threshold $2/(m+2)$; the columns represent values of $m=1,2,3$. Shaded band: $\pm1$ SD over $M=100$ repetitions; error bars: $\pm1$ SE of the mean.}
\label{fig:consistency-grid}
\end{figure}

The results in Figure~\ref{fig:consistency-grid} are all for $k=1$; Figure~\ref{fig:consistency-kcompare} checks whether the same convergence holds for $k=2$ and $k=3$, at a single representative $\kappa$ (35\% of the variance-existence threshold, common to all three dimensions so that the comparison across $m$ is on equal footing). In all the plots, this trend remains consistent: in each panel, the mean of $T_{N,k}$ starts from a value nonzero at $N=500$ and converges to a narrow range around zero at $N=5000$; however, this process occurs at a rate that is not visibly dependent on $k$. What does change with $k$ is the width of the shaded band, which narrows slightly moving from $k=1$ to $k=3$ -- consistent with $k$ entering the variance term of Lemma~\ref{lem:marginal} only through a bounded constant rather than through the $O(n^{-1})$ rate itself. Together with Figure~\ref{fig:consistency-grid}, this indicates that the choice of $k$ mainly affects the variance of the estimator rather than whether or how fast it converges.

\begin{figure}[H]
\centering
\includegraphics[width=\textwidth]{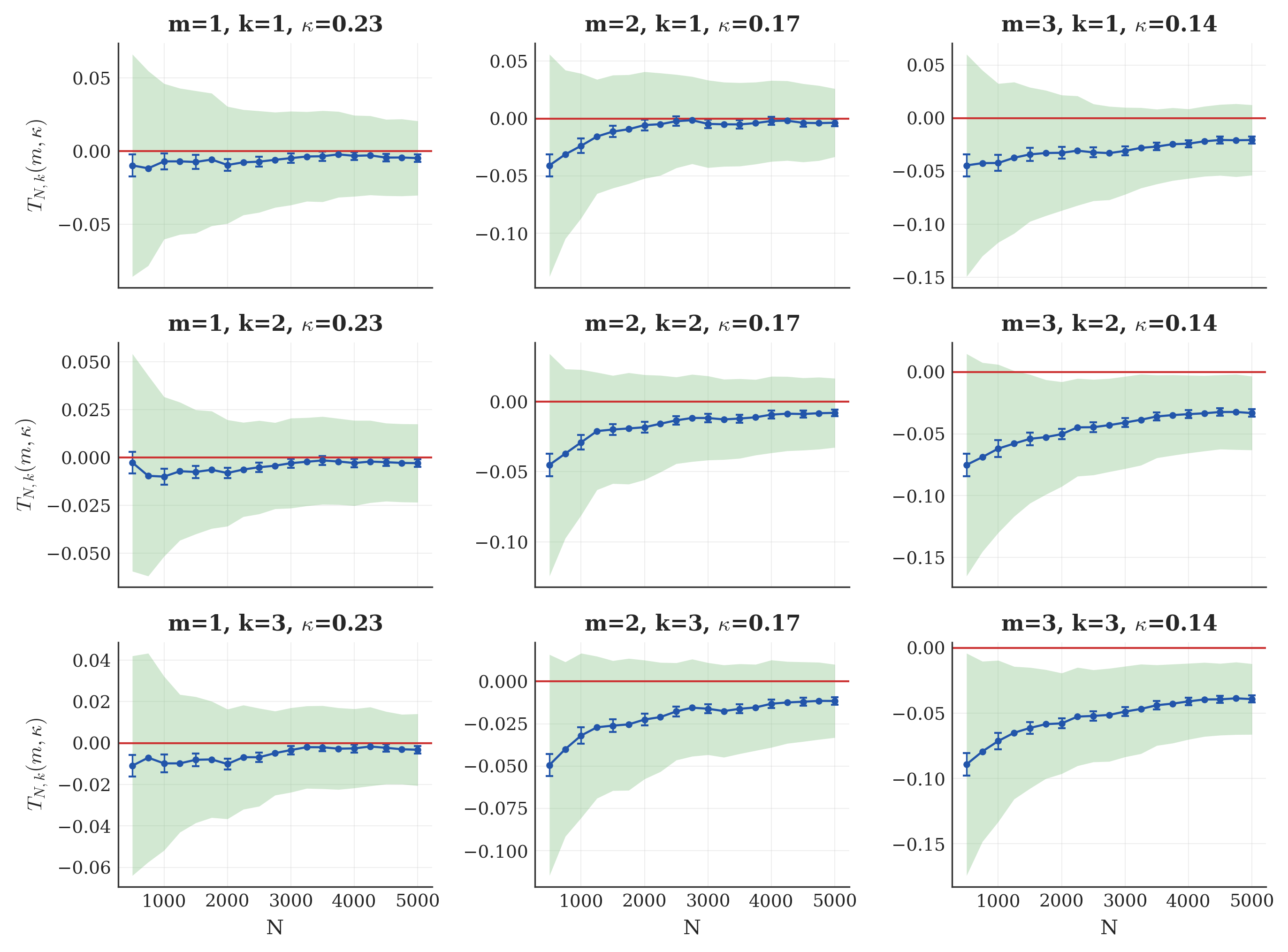}
\caption{$T_{N,k}(m,\kappa)$ converges to zero as $N$ increases; for neighborhood sizes $k=1,2,3$ (rows) and the dimensions $m=1,2,3$ (columns); for each value of $m$, the variance-existence threshold of $\kappa$ has been held constant at 35\% of $2/(m+2)$. Shaded band: $\pm1$ SD over $M=100$ repetitions; error bars: $\pm1$ SE of the mean.}
\label{fig:consistency-kcompare}
\end{figure}

Figures~\ref{fig:consistency-grid} and~\ref{fig:consistency-kcompare} only establish that the bias shrinks; they do not say at what rate. Theorem \ref{thm:consistency} says that
$\mathbb E[(\widehat S_{\kappa,n} -S_\kappa[f])^2] = O(n^{ -2\gamma/m}+n^{-1}).$ When the density is uniform, the $n^{-1}$ piece is the main part. As a result, $\mathrm{RMSE}(T_{N,k})$ goes down like $N^{-1/2}$. On the log log plot, it goes down at a rate of $N^{-1/2}$. Figure~\ref{fig:convergence-rate} checks this directly, for $m=2$, $\kappa=0.25$ and $k=3$, by fitting a straight line to $\log\,\mathrm{RMSE}$ against $\log N$ over 12 sample sizes between 100 and 4000 (80 replications each). The slope from our fit is -0.495. In real terms, that is close to  -1/2. The second plot also looks at bias by itself. There, the 95\% band shifts closer to zero. As $N$ increases, the 95\% band shifts toward the origin, and the 95\% band is not as wide as the 95\% band in the log-log plot.

\begin{figure}[H]
\centering
\includegraphics[width=\textwidth]{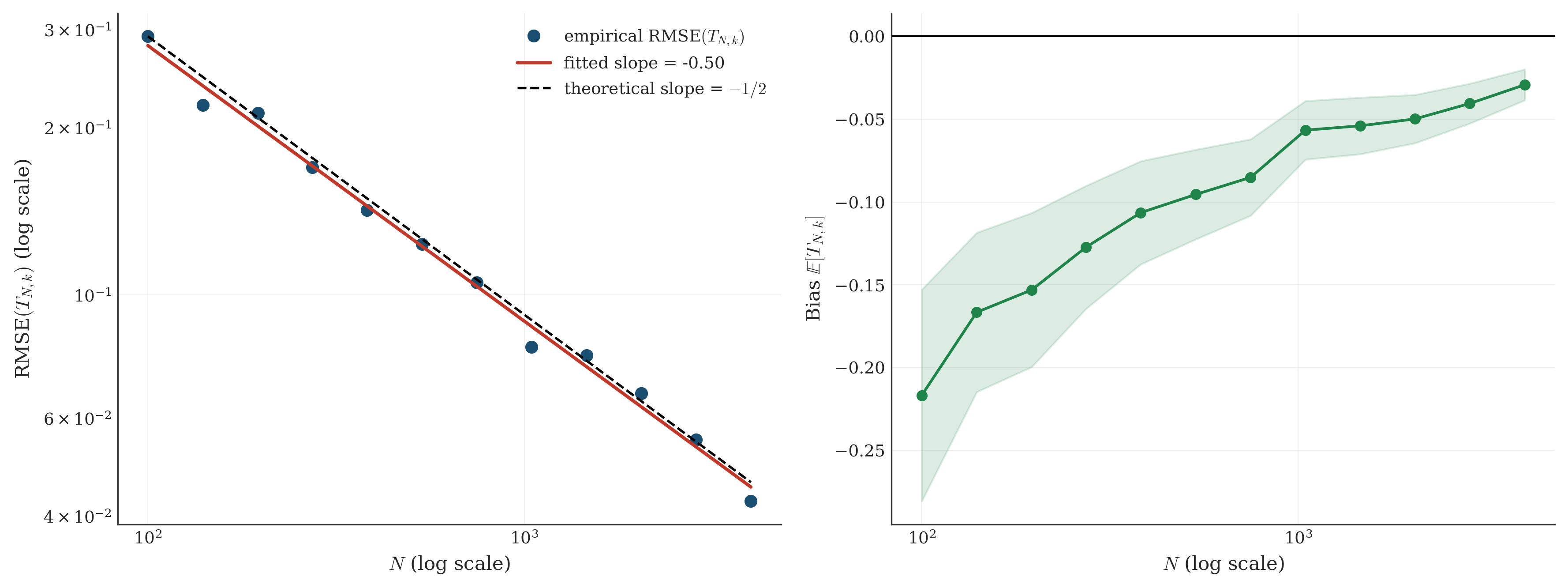}
\caption{Left: $\mathrm{RMSE}(T_{N,k})$ against $N$ on log-log axes, for $m=2$, $\kappa=0.25$, $k=3$, with a fitted power-law slope compared against the theoretical $-1/2$ exponent. Right: bias of $T_{N,k}$ against $N$ with a 95\% confidence band, on a log $N$ axis. Based on 80 replications per sample size.}
\label{fig:convergence-rate}
\end{figure}

One of the free parameters remaining in the estimator is $k$, representing the number of nearest neighbors used to compute $\widehat{S}_{\kappa,n}$; Figure~\ref{fig:consistency-heatmap} illustrates the effect of this parameter, along with $N$, on $\mathrm{RMSE}(T_{N,k})$ for parameters $m=2$ and $\kappa=0.25$, with $k$ ranging from $1$ to $5$, in a single heatmap.  Both axes are important, although not to the same degree: moving along the $N$ axis while maintaining a constant $k$ reduces the RMSE by approximately a factor of five over the range considered, whereas moving along the $k$ axis while maintaining a constant $N$ produces a relatively modest improvement, most of which has already been achieved at $k=2$ or $3$. The fact that the variance term in Lemma~\ref{lem:marginal} is $O(n^{-1})$—regardless of $k$ (for a fixed $k$)—is consistent with this; here, $k$ exerts its influence solely through the constant.

\begin{figure}[H]
\centering
\includegraphics[width=0.85\textwidth]{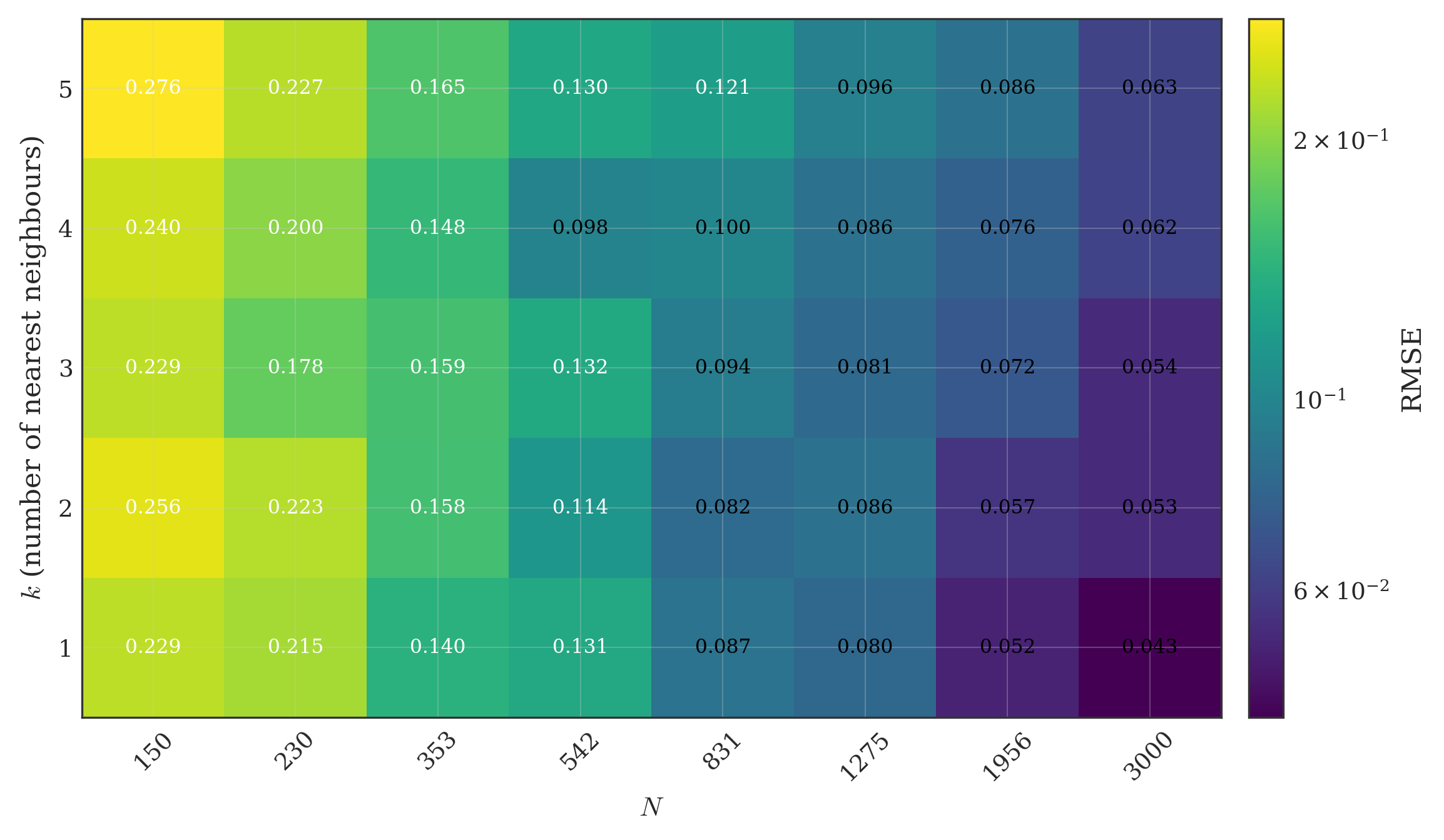}
\caption{$\mathrm{RMSE}(T_{N,k})$ across the joint $(k,N)$ grid, for $m=2$, $\kappa=0.25$, on a logarithmic colour scale. Each cell is based on 40 replications.}
\label{fig:consistency-heatmap}
\end{figure}

\subsection{Asymptotic normality}
\label{sec:sim-normality}

Theorem~\ref{thm:clt} yields $\sqrt n\,T_n\xrightarrow{d}\mathcal N(0,\sigma_{\kappa_0}^2)$; this ensures that the critical values in Section~\ref{sec:sim-critical} are meaningful for a finite number of $N$ values. Figure~\ref{fig:normality} verifies this simultaneously from three perspectives for the values $m=2$, $\kappa=0.25$, and $k=2$: the left panel displays the distribution of $T_{N,k}$ for six sample sizes ranging from $N=150$ to $N=2400$ (each with 400 repetitions); these plots visibly narrow as $N$ increases and converge toward the center of the zero line; the middle panel is a QQ plot of the standardized statistic against the standard normal distribution at the largest $N$ value, shown with a 95\% bootstrap interval, and it shows no systematic deviation from the diagonal; and the right panel tracks the Shapiro-Wilk $p$-value for all six sample sizes; this value remains comfortably above 0.05 throughout—the normality of $T_{N,k}$ has not been rejected for any of the sample sizes we examined.

\begin{figure}[H]
\centering
\includegraphics[width=\textwidth]{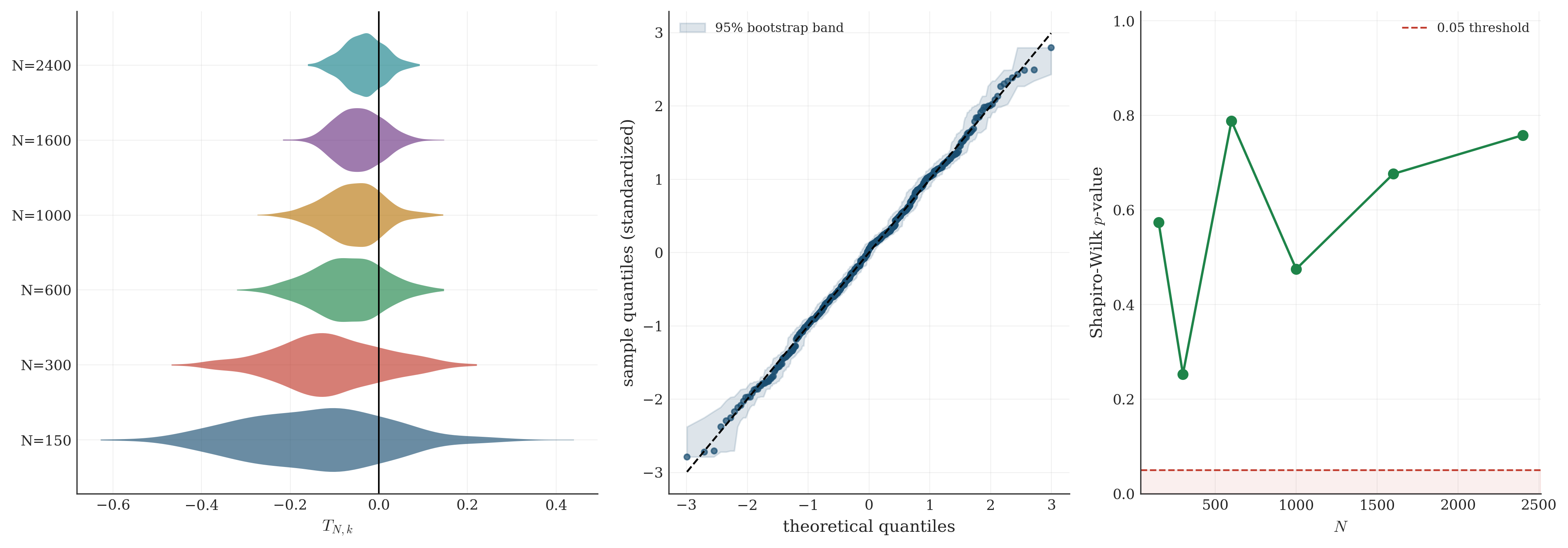}
\caption{Three diagnostic methods regarding the asymptotic normality of $T_{N,k}(2,0,25)$ for $k=2$: (left) The $T_{N,k}$ distribution over the range $N=150,\dots,2400$ is shown as horizontal violin plots (each with 400 repetitions); (center) A QQ plot at $N=2400$ showing the 95\% bootstrap interval relative to $N(0,1)$; (right) the Shapiro-Wilk $p$-value across sample sizes; the 0.05 threshold is shaded.}
\label{fig:normality}
\end{figure}

\subsection{Critical values}
\label{sec:sim-critical}

Table \ref{tab:critical-values} provides the empirical 5\% ciritical value $c_N(0.05)$ for $T_{N,k}(m,\kappa)$. This corresponds to the lower 5\% percentile, consistent with the rejection rule in \eqref{eq:rejection-rule}. The values were obtained from $M=1000$ Monte Carlo simulations. For $\kappa\in\{0.1,0.2,0.3\}$, we list the results for cases where $N$ ranges from 100 to 1000. Here, $m$ takes values from the set $\{2,3\}$, and $k$ is from the set $\{1,2,3\}$. Two main conclusions can be drawn from this. First, when $m$ and $k$ are held constant, the critical value decreases as $N$ increases, and this decrease appears to be on the order of $N^{-1/2}$. This is consistent with what was previously suggested in Section~\ref{sec:sim-consistency}. Second, changing the value of $\kappa$ from 0.1 to 0.3 increases the threshold value, but the magnitude of this increase depends on $m$. When $N=100$, for $k=1$, the increase is approximately doubled when $m=2$. When $m=3$, the increase approaches four times the original value. The reason for this is that the value $\kappa=0.3$ accounts for a larger proportion at the $m=3$ threshold compared to the $m=2$ threshold. For $m=3$, the threshold ratio is $2/5=0.4$, whereas for $m=2$, this ratio is $2/4=0.5$.

\begin{table}[H]
\centering
\caption{Critical values of statistic $T_{N,k}(m,\kappa)$ corresponding to significance level 0.05.}
\label{tab:critical-values}
\begin{tabular}{cccccccc}
\toprule
$\kappa$ & $N$ & \multicolumn{3}{c}{$m=2$} & \multicolumn{3}{c}{$m=3$} \\
 & & $k=1$ & $k=2$ & $k=3$ & $k=1$ & $k=2$ & $k=3$ \\
\midrule
0.1 & 100 & -0.32507 & -0.29025 & -0.28417 & -0.43986 & -0.38162 & -0.37253 \\
 & 200 & -0.24829 & -0.20901 & -0.20595 & -0.31825 & -0.28017 & -0.27527 \\
 & 300 & -0.19255 & -0.16822 & -0.16897 & -0.25294 & -0.24339 & -0.23832 \\
 & 400 & -0.17267 & -0.14033 & -0.14465 & -0.23803 & -0.21382 & -0.21578 \\
 & 500 & -0.16131 & -0.13470 & -0.12815 & -0.20527 & -0.19332 & -0.19235 \\
 & 600 & -0.13368 & -0.11724 & -0.11571 & -0.17299 & -0.17591 & -0.17695 \\
 & 700 & -0.13732 & -0.11703 & -0.11229 & -0.18131 & -0.17070 & -0.16879 \\
 & 800 & -0.11839 & -0.10385 & -0.09941 & -0.16016 & -0.15327 & -0.15412 \\
 & 900 & -0.11581 & -0.09795 & -0.09214 & -0.15561 & -0.14360 & -0.15116 \\
 & 1000 & -0.11029 & -0.09654 & -0.09337 & -0.14535 & -0.13984 & -0.14095 \\
\addlinespace
0.2 & 100 & -0.43606 & -0.39029 & -0.38400 & -0.69080 & -0.66058 & -0.66543 \\
 & 200 & -0.30201 & -0.26918 & -0.27466 & -0.48439 & -0.49963 & -0.51938 \\
 & 300 & -0.22942 & -0.22285 & -0.22538 & -0.44402 & -0.43730 & -0.44894 \\
 & 400 & -0.22036 & -0.20379 & -0.20796 & -0.37510 & -0.39496 & -0.40263 \\
 & 500 & -0.19608 & -0.17471 & -0.18274 & -0.34541 & -0.34900 & -0.36387 \\
 & 600 & -0.16865 & -0.15645 & -0.15931 & -0.30775 & -0.30723 & -0.32858 \\
 & 700 & -0.16162 & -0.14985 & -0.15669 & -0.28305 & -0.29732 & -0.31432 \\
 & 800 & -0.15470 & -0.14364 & -0.14593 & -0.27784 & -0.28175 & -0.29966 \\
 & 900 & -0.14241 & -0.13706 & -0.13259 & -0.25850 & -0.26041 & -0.27449 \\
 & 1000 & -0.13155 & -0.12674 & -0.12894 & -0.24431 & -0.25369 & -0.26777 \\
\addlinespace
0.3 & 100 & -0.67584 & -0.65840 & -0.68309 & -1.76098 & -1.80143 & -1.83409 \\
 & 200 & -0.47427 & -0.47495 & -0.48785 & -1.33379 & -1.43020 & -1.49057 \\
 & 300 & -0.40858 & -0.40549 & -0.41586 & -1.17204 & -1.24295 & -1.31296 \\
 & 400 & -0.36448 & -0.36091 & -0.38181 & -1.07536 & -1.16115 & -1.21465 \\
 & 500 & -0.31646 & -0.31672 & -0.33511 & -0.96859 & -1.06162 & -1.11628 \\
 & 600 & -0.29068 & -0.30903 & -0.31982 & -0.93572 & -1.01000 & -1.07759 \\
 & 700 & -0.27366 & -0.28323 & -0.28992 & -0.90510 & -0.98342 & -1.04411 \\
 & 800 & -0.26985 & -0.26638 & -0.27557 & -0.84191 & -0.94674 & -0.99874 \\
 & 900 & -0.24910 & -0.25049 & -0.26199 & -0.83453 & -0.89693 & -0.94866 \\
 & 1000 & -0.23419 & -0.24365 & -0.25600 & -0.78628 & -0.86819 & -0.92246 \\
\addlinespace
\bottomrule
\end{tabular}
\end{table}

Table~\ref{tab:critical-values} ends at $N=1000$. Therefore, Table~\ref{tab:asymptotic-constants} is used to estimate what will happen for larger values of $N$. Let $\sigma_{m,\kappa,k}$ be the asymptotic standard deviation of $\sqrt{N}\,T_{N,k}(m,\kappa)$. In this study, this value was obtained from $M=1000$ trials conducted for $N=3000$. Additionally, the value $z_{0.05}\approx-1.645$ was used as the 5\% percentile of the standard normal distribution. Then, $\mu_{m,\kappa,k}=\sigma_{m,\kappa,k}\,z_{0.05}$ is determined. Thus, for large values of $N$, the critical value is approximately calculated as $c_N(0.05)\approx \mu_{m,\kappa,k}/\sqrt{N}$,
(see \eqref{eq:rejection-rule}). This means we do not need to repeat the simulation for every value of $N$ we are interested in. Compare this idea with Table~\ref{tab:critical-values}. For example, we consider the values $m=2$, $k=1$, and $\kappa=0.1$. The approximation yields the following result:
$\mu/\sqrt{1000}\approx -2.7633/31.62\approx -0.0874$. The table, however, lists the value $-0.11029$. Thus, while the estimate is not exactly correct when $N=1000$, it does yield a close approximation. This discrepancy is to be expected, since Theorem~\ref{thm:consistency} involves an $O(N^{-1})$ correction. As $N$ increases, the consistency is expected to improve.

\begin{table}[H]
\centering
\caption{Values of $\mu_{m,\kappa,k}$ for $\alpha=0.05$ (empty cells: $\kappa$ exceeds the variance-existence threshold $2/(m+2)$ for that $m$).}
\label{tab:asymptotic-constants}
\begin{tabular}{cccccccccc}
\toprule
$\kappa$ & \multicolumn{3}{c}{$m=1$} & \multicolumn{3}{c}{$m=2$} & \multicolumn{3}{c}{$m=3$} \\
 & $k=1$ & $k=2$ & $k=3$ & $k=1$ & $k=2$ & $k=3$ & $k=1$ & $k=2$ & $k=3$ \\
\midrule
0.05 & -2.5276 & -1.9834 & -1.7300 & -2.9323 & -2.2702 & -2.0532 & -3.3859 & -2.7074 & -2.5572 \\
0.1 & -2.7633 & -2.0950 & -1.7871 & -3.1300 & -2.4534 & -2.2059 & -3.6531 & -2.9601 & -2.6967 \\
0.15 & -2.8170 & -2.0229 & -1.7381 & -3.2049 & -2.6680 & -2.4434 & -4.0497 & -3.3444 & -3.1790 \\
0.2 & -2.7793 & -2.1304 & -1.8713 & -3.5653 & -3.0305 & -2.7998 & -5.1286 & -4.4577 & -4.1955 \\
0.25 & -2.8606 & -2.2116 & -1.9338 & -4.3626 & -3.7113 & -3.4203 & -7.1208 & -6.3112 & -5.9615 \\
0.3 & -3.1265 & -2.4226 & -2.1453 & -5.8019 & -4.8675 & -4.6213 & -13.5903 & -12.0427 & -11.2679 \\
\bottomrule
\end{tabular}
\end{table}

\subsection{GOF signature across the parameter grid}
\label{sec:sim-gof-heatmap}

These critical value tables are generated by selecting a $(m,\kappa,k)$ combination each time. It is also useful to observe how the value of $T_{N,k}(m,\kappa_0)$ changes when you vary two variables simultaneously. On the one hand, there is the $\kappa_0$ estimate; on the other hand, there is the actual $\kappa_1$ value obtained from the data. Figure~\ref{fig:gof-heatmap} illustrates this in a single view. Here, we set $m=2$, $N=2500$, and $k=3$. We then scan both parameters across the same seven-point grid. The grid extends from 12\% to 70\% of the threshold. When $\kappa_0=\kappa_1$, that is, along the diagonal line, the statistic remains close to zero over a large portion of the graph. As we move away from this diagonal, the values drop even further below zero. As we move toward larger $\kappa$ values, this decline becomes more pronounced. This pattern is consistent with what you observed in the power analysis in Section~\ref{sec:sim-power-within}.

\begin{figure}[H]
\centering
\includegraphics[width=0.8\textwidth]{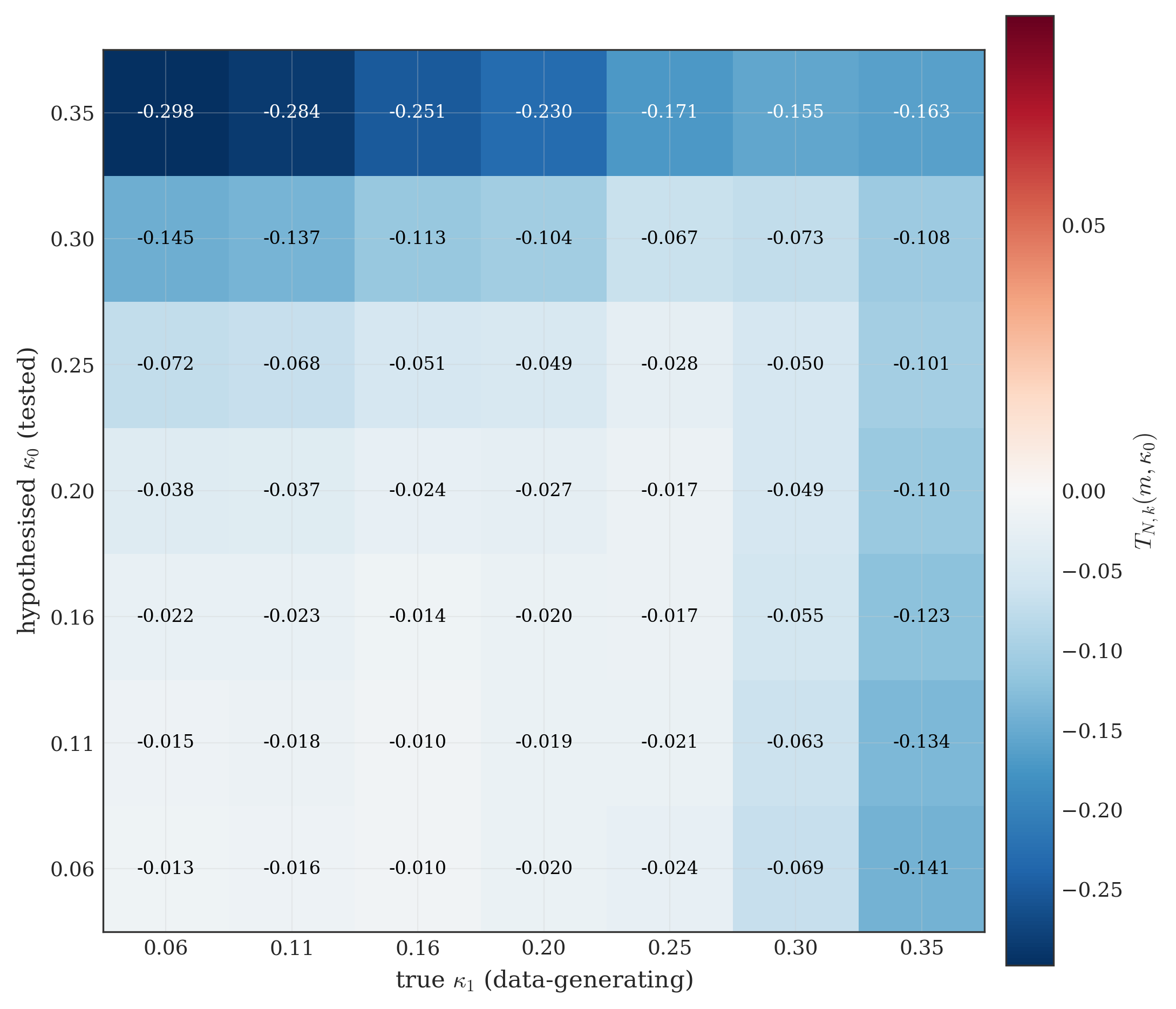}
\caption{$T_{N,k}(m,\kappa_0)$ for $m=2$, $N=2500$, $k=3$, as a function of the hypothesised $\kappa_0$ (rows) and the true data-generating $\kappa_1$ (columns), each cell averaged over 60 replications.}
\label{fig:gof-heatmap}
\end{figure}

One cell in this table deserves closer examination on its own. Figure~\ref{fig:threshold-bias}, isolates the behavior of $T_{N,k}(m,\kappa_0)$ under the \emph{correct} null hypothesis ($\kappa_1=\kappa_0$) as the value of $\kappa_0$ is increased from 50\% to 92\% of the variance-existence threshold, for a fixed $N=2500$. The bias remains moderate up to approximately 75\% of the threshold and then increases sharply, reaching $-1.7$ at 92\% — we attribute this not to the $1/\kappa^2$ term in Theorem~\ref{thm:consistency} alone (see the explanation following that theorem). However, the fact that the density becomes increasingly difficult to estimate locally as $\E[R^2]$ diverges near the limit, rather than at a proven rate. In practice, this means that the nominal critical values in Table~\ref{tab:critical-values} are reliable primarily in cases where $\kappa$ lies comfortably within the acceptable range; furthermore, we do not recommend applying the test, as calibrated here, to a default $\kappa_0$ value exceeding approximately 70–75\% of $2/(m+2)$ without either a significantly larger $N$ value or a resimulated critical value at that specific $\kappa_0$ value.

\begin{figure}[H]
\centering
\includegraphics[width=0.75\textwidth]{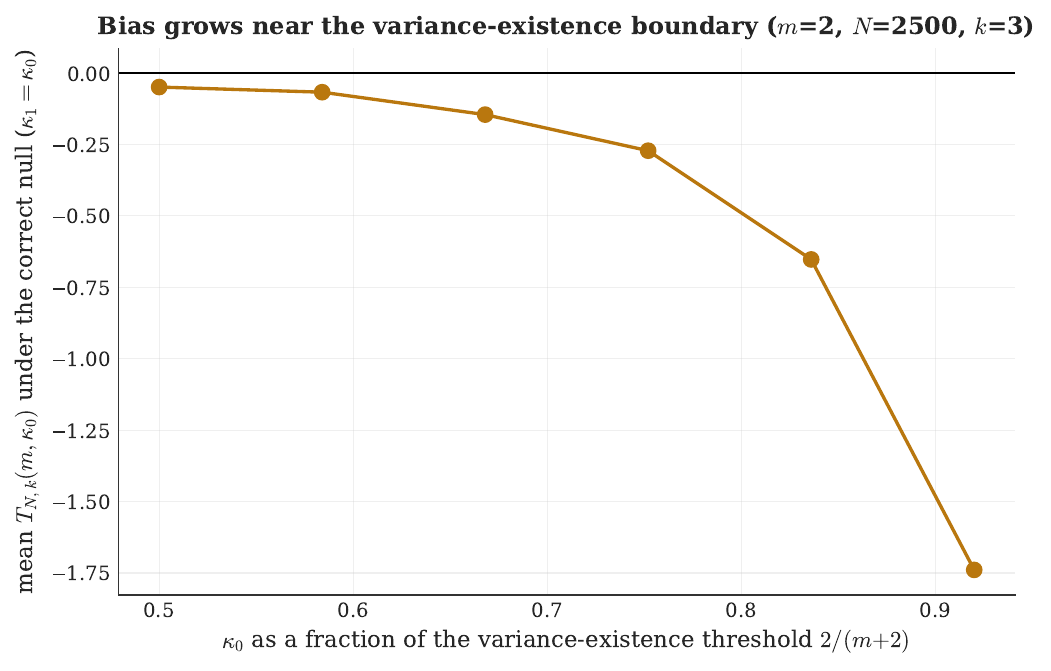}
\caption{Mean of $T_{N,k}(m,\kappa_0)$ under the correct null as $\kappa_0$ approaches the variance-existence threshold $2/(m+2)$, for $m=2$, $N=2500$, $k=3$ (30 replications per point).}
\label{fig:threshold-bias}
\end{figure}

\subsection{Power against within-family alternatives}
\label{sec:sim-power-within}

Before we generalize for the values of $m$, $N$, and $k$, it is useful to examine a single representative example in detail. Let us consider Figure~\ref{fig:power-single}$m=2$, $N=1500$, $k=3$, and $\kappa_0=0.225$ (45\% of the threshold). We estimate the empirical 5\% critical value from $M=500$ repetitions of the null hypothesis and then calculate the rejection rate based on a grid of true $\kappa_1$ values ranging from 5\% to 92\% of the threshold. At $\kappa_1=\kappa_0$, the power is nearly exactly 0.05 due to the nature of the critical value; when $\kappa_1$ exceeds $\kappa_0$, the power increases sharply and reaches 1 long before reaching the top of the grid; in contrast, when $\kappa_1 < \kappa_0$, the power barely rises above the nominal level.

\begin{figure}[H]
\centering
\includegraphics[width=0.75\textwidth]{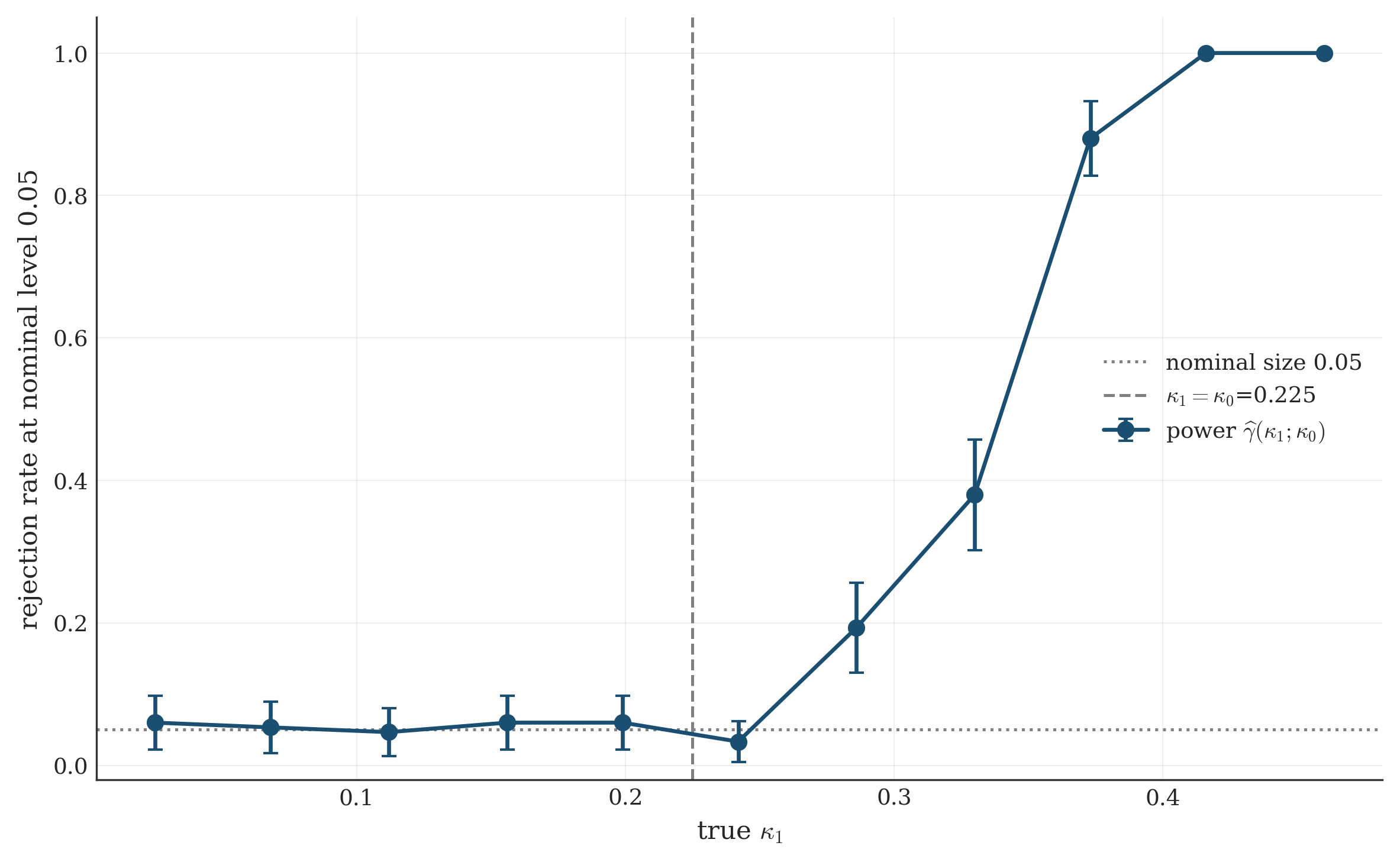}
\caption{Empirical power curve for $H_0:\kappa=\kappa_0=0.225$ against $\kappa GG(2,\kappa_1)$ alternatives, at $N=1500$, $k=3$. The critical value is the empirical 5th percentile of 500 null replications; power at each $\kappa_1$ is based on 150 replications.}
\label{fig:power-single}
\end{figure}

As shown in Figure~\ref{fig:power-within}, the same process is repeated for all combinations of $m$, $N$, and $k$, while keeping $\kappa_0 = 0.20$ constant. The pattern observed in the single case above holds everywhere: As the value of $\kappa_1$ rises above $\kappa_0$, the power increases rapidly and uniformly across all values of $m$, $N$, and $k$; as the value of $\kappa_1$ falls below $\kappa_0$, the power remains close to the nominal 5\% level and shows only a slight increase even at the smallest simulated value of $\kappa_1$. This asymmetry does not stem from a deviation in the critical value estimate. The mean of $T_{N,k}$, for $\kappa_1 < \kappa_0$, as already seen in Figure~\ref{fig:gof-heatmap}, changes only slightly at the same distance $|\kappa_1 - \kappa_0|$, compared to the much larger deviations observed when $\kappa_1 > \kappa_0$   and we interpret this not as a flaw in the estimator but as a genuine property of the $\kappa$-entropy functional: the mapping $\kappa\mapsto S_\kappa[f_\kappa]$ is less sensitive to zero on the $\kappa_0$ side compared to the side approaching the variance-to-variance limit. Therefore, instead of a two-sided test on $|T_n|$, we adopted the one-sided rejection rule given in Section~\ref{sec:test} as \eqref{eq:rejection-rule}: A two-sided test conducted using the same nominal level of significance would expend approximately half the probability of rejection in the direction of the statistic that contains the least information. Depending on this, a one-sided test is quite suitable for detecting alternatives with heavier tails than the hypothesized $\kappa_0$; conversely, it is not very suitable for detecting deviations with lighter tails; if alternatives with lighter tails are of practical interest, a truly two-sided or variance-stabilized statistic would be a natural direction for future work.

\begin{figure}[H]
\centering
\includegraphics[width=\textwidth]{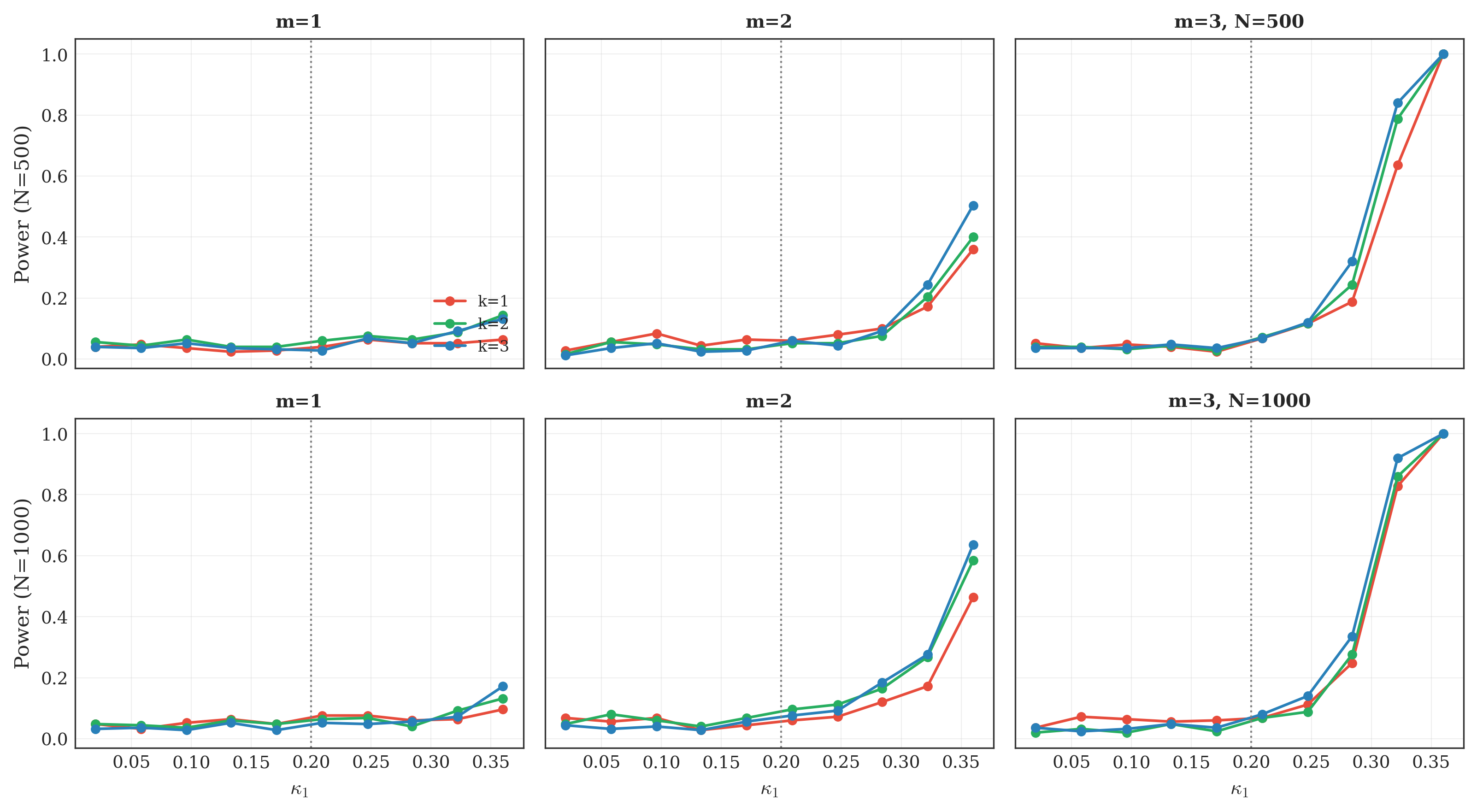}
\caption{Empirical power of the test for $\kappa_0=0.20$ against $\kappa GG(m,\kappa_1)$ alternatives, $M=250$--500 repetitions. Columns: $m=1,2,3$; rows: $N=500,1000$; colours: $k=1,2,3$. The vertical dotted line marks $\kappa_1=\kappa_0$.}
\label{fig:power-within}
\end{figure}

\subsection{Power against Student-\emph{t} alternatives}
\label{sec:sim-power-cross}
In Figure~\ref{fig:power-cross}, we consider the null hypothesis $\kappa_0 = 0.05$, which is close enough to the Gaussian limit for the comparison to be meaningful, and evaluate the power against multivariate Student-$t_\nu$ alternative hypotheses as the value of $\nu$ varies between 3 and 50. Here, we see that as $\nu\to\infty$, the results align with what would be expected from the limit relationship $t_\nu\to N(0,I_m)$: the power is close to one or near one for small $\nu$ values, where the Student-$t$ distributions are most concentrated and can be better distinguished from $\kappa GG(m,0.05)$, and it decreases monotonically as $\nu$ increases; it stabilizes at a level close to the nominal power when $\nu$ exceeds approximately 20. For larger $m$ values, the decrease is slightly faster, and $k=3$ consistently remains at or above the value of $k=1$ throughout the entire curve; this reflects the ranking already visible in Figure~\ref{fig:consistency-heatmap}.

\begin{figure}[H]
\centering
\includegraphics[width=\textwidth]{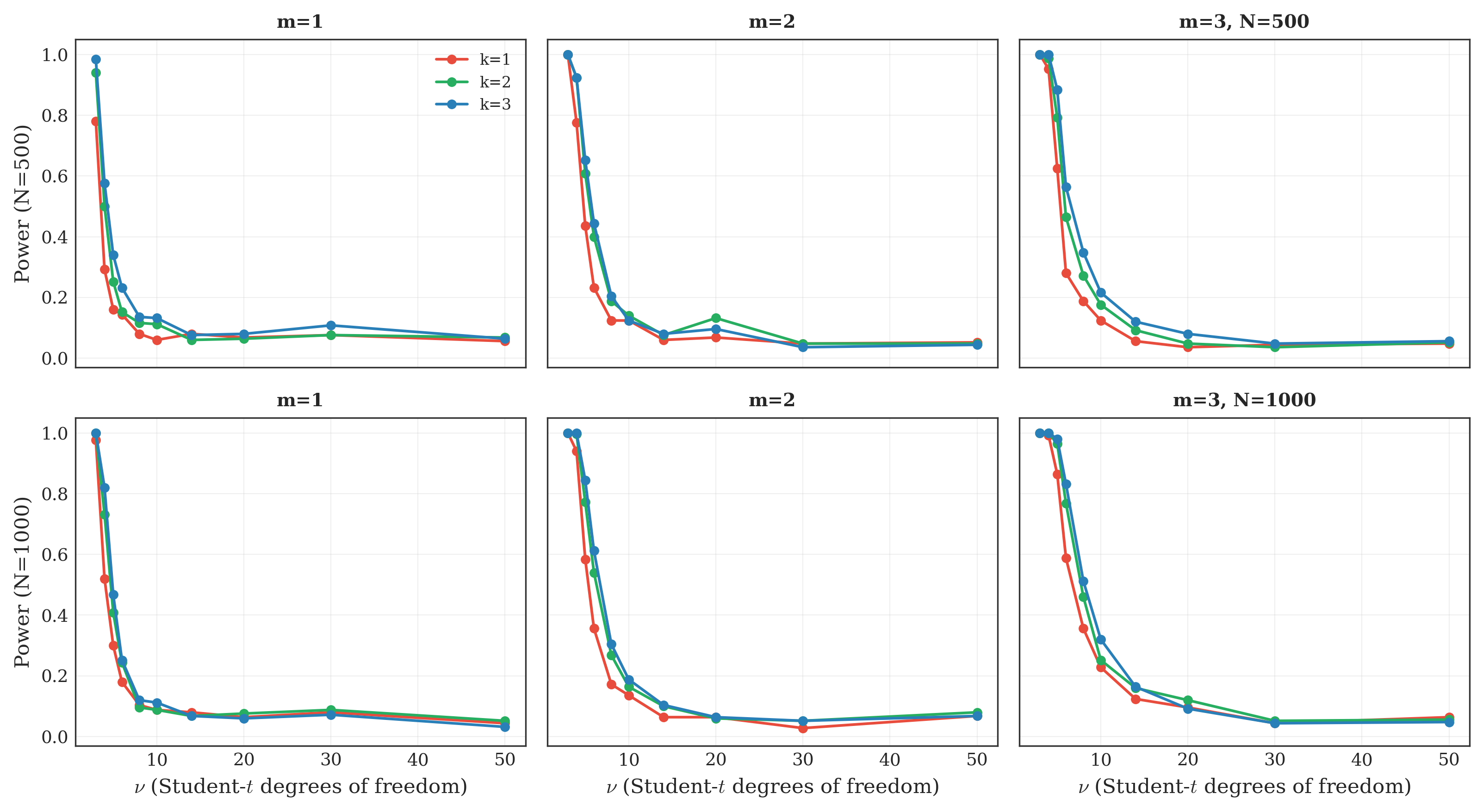}
\caption{Empirical power of the test for $\kappa_0=0.05$ against multivariate Student-$t_\nu$ alternatives, $M=250$--500 repetitions. Columns: $m=1,2,3$; rows: $N=500,1000$; colours: $k=1,2,3$.}
\label{fig:power-cross}
\end{figure}

When evaluated as a whole, these results support the theoretical developments in Sections $\ref{sec:estimator}$–$\ref{sec:clt}$ regarding the simple null hypothesis test in Section $\ref{sec:test}$: the estimator approaches the expected value at a rate of $N^{-1/2}$, the finite-sample distribution approaches the normal distribution long before $N$ reaches the thousands, the critical values behave consistently with the $1/\kappa^2$ dependence implied by the consistency theorem for small $\kappa$ values, and for large $\kappa$ values, they break down near the variance-presence limit by a mechanism that we have empirically documented but not derived; the resulting one-sided test possesses significant power against both within-family and between-family alternative hypotheses in the direction for which it was designed. Taking a broader view, these complete the natural quartet of entropy-based GOF tests—Shannon (Gauss), Rényi (Student/Pearson II), Tsallis ($q$-Gauss), and Kaniadakis ($\kappa$ -Gauss)—all of which are constructed from the same combination of the maximum entropy principle and the $k$-NN estimator, and the $\kappa$-Gauss case behaves consistently with this model despite the estimator’s different algebraic structure.

\section{Conclusion}
\label{sec:conclusion}

This article extends the entropy-based nearest-neighbor approach for goodness-of-fit testing to the Kaniadakis entropy, the sole member of the standard single-parameter entropy deformations that has not yet been included in this framework. Starting from the maximum entropy principle, we derived the multivariate $\kappa$-Gaussian distribution—including a closed-form expression for the scaling parameter $\beta(\kappa,m)$ obtained via a scaling relation for the radial normalization integral—and, through the scaling identity of the $\kappa$ -exponential function, for which there is no clear analog to the Tsallis $q$-exponential function; these elements, when combined, ensure that the resulting test statistic is not only definable but also fully computable. Instead of a single functional sufficient for R’enyi or Tsallis entropy, we showed that the corresponding $k$-nearest-neighbor estimator, constructed as the difference between two power-sum functionals, is $L^2$-consistent at the standard rate of $N^{-1/2}$ and asymptotically normal, and we determined a precise threshold value: We determined that $\kappa < 2/(m+2)$; below this threshold, the $\kappa$-Gaussian distribution has finite variance—a condition not addressed in previous papers in this series, and one that has proven to be significant not only in theory but also in practice.

The simulation study confirms the theory and presents findings that cannot be clearly deduced from asymptotic analysis alone. First, the estimator’s behavior on finite samples breaks down near the limit where variance exists through a mechanism that is different from and greater than the $1/\kappa^2$ term that determines the small $\kappa$ in Theorem~\ref{thm:consistency}. We document this empirically in Section~\ref{sec:sim-gof-heatmap}, but we have not derived a ratio for this and note it as an open technical question. Second, because of the test’s structure, we adopted a one-sided rejection rule (Section~\ref{sec:test}), since the mapping $\kappa\mapsto S_\kappa[f_\kappa]$ is much more sensitive to alternatives with heavier tails than to those with lighter tails; which detects heavier-tailed deviations well but has low power against lighter-tailed deviations; if both directions are of practical interest, a two-sided or variance-stabilized alternative is a natural next step. Third, we have mentioned but not derived the plug-in/sandwich correction required to extend the test from the simple null hypothesis examined here (known $\mu,\Sigma$) to the composite null hypothesis (estimated $\mu,\Sigma$) that most applications will require; which, in our view, is the most urgent piece of unfinished work left by the paper.

This article is intended to serve as a template that goes beyond the specific test developed here: The same combination of maximum entropy characterization, a $k$-NN entropy estimator, and a zero-entropy term in closed form can be assembled for any entropy functional that performs sufficiently well; The $\kappa$-Gaussian case, meanwhile, concretely demonstrates what changes when the functional is not a sum of single powers. The next natural candidate is the \emph{Sharma–Mittal entropy} \citep{sharma1975}; this is a two-parameter family that includes the R\'enyi and Tsallis entropies as special or limiting cases (and Shannon entropy as a further limit beyond that). Deriving the Sharma–Mittal maximum entropy distribution and its NN-based GOF test will encompass \citet {cadirci2025renyi,cadirci2025tsallis} as special cases within a single framework; whether it can similarly encompass the $\kappa$-Gaussian distribution established here, or whether Kaniadakis entropy is truly an orthogonal deformation rather than a special case of something larger, is an open question that must be addressed directly in the upcoming work.

\section*{Code Availability}

The code implementing the $\kappa$-Gaussian distribution, the sampler, the $k$-NN entropy estimator, and the full simulation study (consistency, critical values, asymptotic normality, and power against within-family and Student-$t$ alternatives) is available at:
\begin{center}
\url{https://github.com/mehmetsiddik/kappa-gaussian-gof-test}
\end{center}

\end{document}